\documentclass[11pt,a4paper]{article}
\usepackage{amsthm,amsmath,amssymb}
\usepackage{tikz,xcolor,graphicx}
\usepackage{mathtools}
\usepackage{enumitem}
\usepackage{booktabs}
\usepackage{float}
\usepackage{url}
\usepackage{aligned-overset}
\usepackage[T1]{fontenc}

\newcounter{localclaim}

\usepackage{caption}

\usepackage{booktabs}

\usetikzlibrary{calc}
\usetikzlibrary{positioning}

\usepackage{graphicx}
\usepackage{subfig}

\usepackage{hyperref}

\usepackage[noabbrev,capitalize]{cleveref}

\crefname{equation}{}{}

\newtheorem{theorem}{Theorem}[section]

\newtheorem{lemma}{Lemma}[section]

\theoremstyle{definition}
\newtheorem{definition}[theorem]{Definition}

\newtheorem{claimthm}{Claim}[section]
\newenvironment{claimproof}
  {\proof}
  {\endproof}
  {\endproof}

\newcommand{\floor}[1]{\left\lfloor #1 \right\rfloor}
\newcommand{\ceil}[1]{\left\lceil #1 \right\rceil}

\makeatletter
\newcommand*\bigcdot{\mathpalette\bigcdot@{.5}}
\newcommand*\bigcdot@[2]{\mathbin{\vcenter{\hbox{\scalebox{#2}{$\m@th#1\bullet$}}}}}
\makeatother

\DeclareMathOperator{\ex}{ex}

\begin{document}

    \title{The Tur\'{a}n number of the Cartesian product of a star and an edge
    }
    \author{
    Xiamiao Zhao\thanks{Department of Mathematical Science, Tsinghua University, Beijing 100084, China. Email: zxm23@mails.tsinghua.edu.cn} \qquad
    Xin Cheng\thanks{Corresponding author. School of Mathematics and Statistics, Northwestern Polytechnical University and Xi'an-Budapest Joint Research Center for Combinatorics, Xi'an 710129, Shaanxi, P.R. China. Email: xincheng@mail.nwpu.edu.cn.} \qquad
    Cheng Chi\thanks{School of Mathematical Sciences, Shanghai Jiao Tong University, 800 Dongchuan Road, Shanghai 200240, China. Email: chengchi@sjtu.edu.cn.} \qquad
    Ervin Gy\H{o}ri\thanks{Alfr\'{e}d R\'{e}nyi Institute of Mathematics, Budapest, Hungary. Email: gyori.ervin@renyi.hu} \qquad
    Casey Tompkins\thanks{Alfr\'{e}d R\'{e}nyi Institute of Mathematics, Budapest, Hungary. Email: tompkins.casey@renyi.hu.} \\
    Yichen Wang\thanks{Department of Mathematical Science, Tsinghua University, Beijing 100084, China. Email: wangyich22@mails.tsinghua.edu.cn}}
    \date{}
    \maketitle

    \begin{abstract}
        
        Let $C_k$ denote the cycle of length $k$, $S_t$ be a star with $t$ edges. And
        let $B_t$ be the graph consisting of $t$ copies of $C_4$ sharing one fixed edge.  Equivalently, $B_t=K_2 \mathbin{\square} S_t$, which is the Cartesian product of a star with $t$ edges and an edge.
        Recently, Gao, Janzer, Liu and Xu [\textit{Israel J. Math. 269(2025)}] proved that the Tur\'an number of $K_2\mathbin{\square} C_{2l}$ is $\Theta(n^{\frac{3}{2}})$ for every $l\ge 4$.

        In this paper, we obtain upper and lower estimates for the Tur\'an number of $B_t$ in both the general and bipartite settings for every $t\geq 2$.
        For the lower bound, we use random construction based on the extremal structure of $C_4$.
        These results imply that
        $\frac{1}{2\sqrt{2}}\leq \lim_{t\to \infty} \frac{\ex(n,B_t)}{\sqrt{t}}\leq \frac{1}{2}$, and $\frac{1}{4}\leq \lim_{t\to \infty} \frac{\ex_{bip}(n,B_t)}{\sqrt{t}}\leq \frac{1}{2\sqrt{2}}.$
        In the case of $B_2$, we obtain sharper estimates.
        We show that the Tur\'an number of $B_2$ is approximately between $(0.518+o(1))n^{\frac{3}{2}}$ and $(0.603+o(1))n^{\frac{3}{2}}$.
        And in the bipartite setting, it is approximately between
        $(0.385+o(1))n^{\frac{3}{2}}$ and $(0.468+o(1))n^{\frac{3}{2}}$.
        Moreover,
        in the bipartite setting, we give a more general result, which shows that for every tree $T$ with $t$ edges, the bipartite Tur\'an number of $K_2\mathbin{\square}T$ is at most $\frac{\sqrt{t}}{2\sqrt{2}}(1+o(1))n^{\frac{3}{2}}$.
    \end{abstract}


    \section{Introduction}\label{sec:intro}

        For a given graph $F$, a graph $G$ is said to be {\it $F$-free} if it does not contain $F$ as a subgraph.
        A classical problem in extremal combinatorics is to determine the maximum number of edges in an $F$-free graph on $n$ vertices.
        The maximum number of edges is called the \emph{Tur\'{a}n number} for $F$ and is denoted by $\ex(n, F)$.
        Let $K_p$ denote the complete graph on $p$ vertices.
        The study of Tur\'an-type problems dates back to the early 20th century, when Mantel \cite{mantel1907} determined the Tur\'an number of $K_3$.
        Tur\'an \cite{TuranEgy1941} generalized this result by determining $\ex(n,K_p)$ for all $n$ and $p$.
        Later, Erd\H{o}s, Stone and Simonovits~\cite{ErdosA1966, ErdosOn1946} proved that $\ex(n,F)$ is asymptotically determined by the chromatic number $\chi(F)$ of $F$, when $\chi(F)\ge 3$.  Specifically they proved
        \begin{equation}\label{eq:ess}
            \ex(n,F)=\frac{1}{2}\left(1-\frac{1}{\chi(F)-1}\right)n^2 + o(n^2).
        \end{equation}
        However~\eqref{eq:ess} only shows that $\ex(n, F) = o(n^2)$ when $F$ is bipartite.
        In general it is difficult to determine the exact value, or even the order of magnitude, of the Tur\'an number of a bipartite graph.
        Indeed, there are relatively few bipartite graphs for which the order of magnitude of the Tur\'{a}n number is known.
        Two fundamental classes of bipartite graphs that have been extensively studied are even cycles and complete bipartite graphs.
        For the even cycle $C_{2k}$, Bondy and Simonovits \cite{bondy1974cycles} gave an upper bound $\ex(n, C_{2k}) = O(n^{1+\frac{1}{k}})$.
        For the complete bipartite graphs $K_{s,t}$, a well-known result of K\H{o}v\'{a}ri, R\'{e}nyi and S\'{o}s \cite{kst} showed that $\ex(n, K_{s,t}) \leq \frac{(t-1)^{\frac{1}{s}}}{2}(1+o(1))n^{2-\frac{1}{s}}$ for $s \leq t$ when $n$ is sufficiently large.
        We refer readers who are interested in the Tur\'an number of bipartite graphs to the survey of F\"{u}redi and Simonovits~\cite{FurediThe2013}.

        Let $S_t$ denote the star with $t$ edges.
        The {\it Cartesian product} of $G$ and $H$, denoted by $G \mathbin{\square} H$, is the graph whose vertex set is $\{(g,h): g\in V(G), h\in V(H)\}$ where two vertices $(g_1,h_1)$ and $(g_2,h_2)$ are adjacent if and only if $g_1=g_2$ and $h_1h_2\in E(H)$ or $h_1=h_2$ and $g_1g_2\in E(G)$.

        \begin{figure}
            \centering
            \begin{tikzpicture}
                \filldraw (0,0) circle (.1) node[left] {$u'$};
                \filldraw (1,0) circle (.1) node[right] {$v_1'$}; 
                \filldraw (2,0) circle (.1) node[right] {$v_2'$};
                \filldraw (3,0) circle (.1) node[right] {$v_t'$};
                \filldraw (0,1) circle (.1) node[left] {$u$};
                \filldraw (1,1) circle (.1) node[right] {$v_1$};
                \filldraw (2,1) circle (.1) node[right] {$v_2$};
                \filldraw (3,1) circle (.1) node[right] {$v_t$};
                \filldraw (2.25,0.5) circle (.01);
                \filldraw (2.5,0.5) circle (.01);
                \filldraw (2.75,0.5) circle (.01);
    
                \draw (0,0)--(0,1);
                \draw (0,1)--(1,1);
                \draw (1,1)--(1,0);
                \draw (1,0)--(0,0);
    
                \draw (2,0)--(2,1);
                \draw (3,0)--(3,1);

                \draw (0,0)..controls (0.5,-0.35) and (1.5,-0.35)..(2,0);
                \draw (0,0)..controls (1,-0.7) and (2,-0.7)..(3,0);
                \draw (0,1)..controls (0.5,1.35) and (1.5,1.35)..(2,1);
                \draw (0,1)..controls (1,1.7) and (2,1.7)..(3,1);
            \end{tikzpicture}
            \caption{The graph $B_t = S_t \mathbin{\square} K_2$. The graph can be viewed as $t$ copies of $C_4$ sharing a common edge.}
            \label{fig:Bt}
        \end{figure}
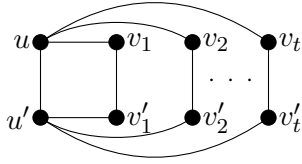
                 Tur\'an-type problems for the Cartesian product of two bipartite graphs, particularly two trees, have attracted much attention.
        Let $P_t$ be the path with $t$ vertices.
        Brada\v{c}, Janzer, Sudakov and Tomon \cite{bradavc2023turan} prove that there exist positive real numbers $c= c(t)$ and $C = C(t)$ such that
        \begin{equation*}
            cn^\frac{3}{2} \leq \ex(n, P_t \mathbin{\square} P_t) \leq Cn^{\frac{3}{2}}.
        \end{equation*}
        They also prove a more general version, that is, there exist positive real numbers $c=c(T,t)$ and $C=C(T,t)$ such that
        \[
        c n^{\frac{3}{2}} \leq \ex(n, P_t \mathbin{\square} T) \leq C n^{\frac{3}{2}}
        \]
        for every tree $T$ and every positive integer $t$.

        In particular, Gao, Janzer, Liu and Xu \cite{GaoExtre2023} improve the upper bound of $\ex(n, P_t \mathbin{\square} P_t)$ by proving that for any given $t$ and sufficiently large $n$
        \begin{equation*}
            \ex(n, P_t \mathbin{\square} P_t) \leq 5 t^{\frac{3}{2}} n^{\frac{3}{2}}.
        \end{equation*}
        In addition, they considered the Cartesian product of an even cycle and an edge, and show that
        \begin{equation*}
            \ex(n, K_2 \mathbin{\square} C_{2l}) = \Theta(n^{\frac{3}{2}})
        \end{equation*}
        for any integer $l \geq 4$.



  In this paper we mainly focus on the graphs $B_t \coloneq K_2 \mathbin{\square} S_t$, equivalently $t$ four-cycles sharing a common edge (see Figure~\ref{fig:Bt}). The order of magnitude of the extremal number is $\Theta(n^{\frac{3}{2}})$ by~\cite{bradavc2023turan}, so we focus on the coefficient of $n^{\frac{3}{2}}$.  
  Note that when $t=1$, we have $B_1 = C_4$.
  The upper bound for $\ex(n, C_4)$ is given by K\H{o}v\'{a}ri-S\'{o}s-Tur\'{a}n theorem \cite{kst}.
  Combining the lower bound constructions given independently by Erd\H{o}s, R'{e}nyi and S'{o}s \cite{ErdosOn1966} and Brown \cite{BrownOn1966}, we have
  \begin{equation*}
      \ex(n, C_4) = \frac{1}{2}(1+o(1))n^{\frac{3}{2}}.
  \end{equation*}
  We establish bounds for general $t$, and in the case of $B_2$ we  obtain more precise results. We also give an improvement for the bipartite Tur\'an number of $K_2\mathbin{\square} T$. Our main results are the following.



        \begin{theorem}\label{upper bound of Bt}
        For any $t\ge 1$ and sufficiently large $n$, we have
            \begin{equation*}
                \ex(n, B_t) \leq \frac{\sqrt{t}}{2} (1 + o(1)) n^{\frac{3}{2}}.
            \end{equation*}
            and
            \begin{equation*}
                \ex(n, B_t) \geq 
                \begin{cases*}
                    \frac{\sqrt{\floor{(t+1)/2}}}{2} (1 + o(1)) n^{\frac{3}{2}}, &\text{ if $t$ is odd, }\\
                     \frac{2s(s+1)(\sqrt{4s^2+5s+1}-2s-1)}{(\sqrt{4s^2+5s+1}-s-1)^{\frac{3}{2}}}(1+o(1))n^{\frac{3}{2}}, &\text{ if $t=2s$. }
                \end{cases*}
            \end{equation*}
        \end{theorem}
        By calculation, the above result implies that
        $$\frac{1}{2\sqrt{2}}\leq \lim_{t\to \infty} \frac{\ex(n,B_t)}{\sqrt{t}}\leq \frac{1}{2}.$$
        Next, we will give a better upper bound in the case of $B_2$.
        Theorem \ref{upper bound of Bt} gives the bound 
        $$  \frac{4(\sqrt{10}-3)}{(\sqrt{10}-2)^{\frac{3}{2}}} (1+o(1))n^{\frac{3}{2}} \leq \ex(n, B_2) \leq \frac{1}{\sqrt{2}}(1+o(1))n^{\frac{3}{2}}.$$
        Where $\frac{4(\sqrt{10}-3)}{(\sqrt{10}-2)^{\frac{3}{2}}}\approx 0.518$ and $\frac{1}{\sqrt{2}}\approx 0.707$.
        In the following result, we further improve the coefficient of the upper bound of $ex(n,B_2)$ from $\frac{1}{\sqrt{2}}\approx 0.707$ to $\frac{2}{\sqrt{11}}\approx 0.603$.
        \begin{theorem}\label{thm: upper and lower bound of B2}
        For sufficiently large $n$, we have
            $$\ex(n,B_2)\leq \frac{2}{\sqrt{11}} (1+o(1)) n^{\frac{3}{2}}.$$
        \end{theorem}

        Let $F$ be a bipartite graph.
        The bipartite Tur\'{a}n number of $F$, denoted by $\ex_{bip} (n, F)$, is the maximum number of edges in an $n$-vertex $F$-free bipartite graph.
        The function $\ex_{bip} (n, F)$ has been extensively studied, see~\cite{AlonTuran2003, SudakovTuran2020, YuanExtre2024} for recent progress.
        For the bipartite Tur\'{a}n number of $C_4$, the upper bound is established by Reiman \cite{ReimanUber1958}, while the lower bound is given by the incidence graph of a projective plane.     
        Thus, we have
        \begin{equation*}
            \ex_{bip}(n,C_4) = \frac{1}{2\sqrt{2}}(1+o(1))n^{\frac{3}{2}}.
        \end{equation*}
        For the bipartite case, we get a more general result.

        \begin{theorem}\label{thm: upper bound of K2 cross Tree}
            Let $T$ be a tree with $t$ vertices. For sufficiently large $n$, we have
            $$\ex_{bip}(n,K_2 \mathbin{\square} T)\leq \frac{\sqrt{t-1}}{2\sqrt{2}}(1+o(1)) n^{\frac{3}{2}}.$$
        \end{theorem}
        \noindent
        Specializing to the tree $T=K_{1,t}$ we obtain an upper bound 
        $$\ex_{bip}(n,B_t) \le\frac{\sqrt{t}}{2\sqrt{2}}(1+o(1))n^{\frac{3}{2}}.$$
        The following result provides a lower bound for $\ex_{bip}(n, B_t)$.

        \begin{theorem}\label{bipartite upper bound of Bt}
        For any $t\ge 1$ and sufficiently large $n$, we have
            \begin{equation*}
                \ex_{bip}(n, B_t) \geq
                \begin{cases*}
                    \frac{\sqrt{\floor{(t+1)/2}}}{2\sqrt{2}}(1 + o(1)) n^{\frac{3}{2}}, \text{ if $t$ is odd,}\\
                    \frac{1}{4} \frac{t(t+2)}{(t+1)^{\frac{3}{2}}}(1+o(1)) n^{\frac{3}{2}}, \text{ if $t$ is even}.
                \end{cases*}
            \end{equation*}
        \end{theorem}
    The above results implies that
    $$\frac{1}{4}\leq \lim_{t\to\infty}\frac{ex_{bip}(n,B_t)}{\sqrt{t}}\leq \frac{1}{2\sqrt{2}}.$$

        When $t = 2$, by Theorem \ref{thm: upper bound of K2 cross Tree} and \ref{bipartite upper bound of Bt}, we have $$ \frac{2}{3\sqrt{3}} (1+o(1)) n^{\frac{3}{2}} \leq \ex_{bip}(n ,B_2) \leq \frac{1}{2}(1+o(1))n^{\frac{3}{2}}.$$
        Where $\frac{2}{3\sqrt{3}}\approx 0.385$.
        In the following result, we further improve the upper bound of $\ex_{bip}(n, B_2)$ from $\frac{1}{2}$ to $0.468$.
        
        \begin{theorem}\label{thm: bipartite upper and lower bound of B2}
        For sufficiently large $n$, we have
            $$\frac{2}{3\sqrt{3}} (1+o(1))n^{\frac{3}{2}}\leq \ex_{bip}(n,B_2)\leq 0.468(1+o(1))n^{\frac{3}{2}}.$$
        \end{theorem}


        The paper is organized as follows.
        In Section \ref{sec:intro}, we introduce the background of the problem.
        In Section \ref{sec:lem and notation}, we give some notation and introduce several useful lemmas, and give the lower bounds for $\ex(n, B_t)$ and $\ex_{bip} (n, B_t)$.
        In Section \ref{sec: Bt}, we give the proofs of Theorem \ref{upper bound of Bt} and Theorem \ref{thm: upper and lower bound of B2}.
        In Section~\ref{sec: tree}, we give the proof of Theorem \ref{thm: upper bound of K2 cross Tree}.
        In Section \ref{sec: Bt bipartite}, we give the proofs of Theorem \ref{thm: bipartite upper and lower bound of B2}.

    \section{Preliminaries}\label{sec:lem and notation}

        In this section, we will give some notation and introduce several useful lemmas, and give the lower bounds for $\ex (n, B_t)$ and $\ex_{bip}(n, B_t)$.

        \subsection{Notation}

        For a graph $G$, we use $V(G)$ and $E(G)$ to denote the vertex set and edge set of $G$, respectively.
        We use $v(G) = |V(G)|$ and $e(G) = |E(G)|$ to denote the number of vertices and edges of $G$, respectively. 
        For a subset $S \subseteq V(G)$, we use $N_G(S)$ (or $N(S)$ if the subscript is clear) to denote the common neighbors of $S$, i.e., $N_G(S) = \{v: uv \in E(G) \text{ for every vertex $u$ in $S$}\}$.
        If $S = \{v\}$, we write $N(v)$ instead of $N(S)$ and denote $N[v] = N(v) \cup \{v\}$.
        We use $d(S)$ to denote the codegree of $S$, i.e., $d(S) = |N(S)|$.
        When $S = \{v\}$, $d(S)$ is the degree of $v$.
        For two sets of vertices $A$ and $B$, let $e(A,B)$ be the number of edges with one endpoint in $A$ and the other endpoint in $B$.
        We say a vertex set $S$ is \emph{good} if $d(S) \geq |S| + 1$.
        For a graph $G$, denote by $N(C_4, G)$ the number of copies of $C_4$ (denoted by $u_1 v_1 u_2 v_2$) for which there exists a good set $S$ of size $t$ such that $\{u_1, u_2\} \subseteq S$ or $\{v_1, v_2\} \subseteq S$.
        For bipartite graph $G$ with parts $V_1$ and $V_2$, denote by $N_i(C_4, G)$ the number of copies of $C_4$ (denoted by $C$) for which there exists a good set $S$ of size $t$ such that $C \cap V_i \subseteq S$.
        Let $G[S]$ be the subgraph of $G$ induced by the vertex set $S$.




        \subsection{Lower bounds of $\ex(n, B_t)$ and $\ex_{bip}(n, B_t)$}\label{subsec: lower bound}

        In this subsection, we give the lower bounds of $\ex(n, B_t)$ and $\ex_{bip}(n, B_t)$.
        The following definition is crucial for constructing $B_t$-free graphs.

        \begin{definition}\label{def:blow up}
            Let $G$ be a graph and $t$ be a positive integer.
            The {\it $t$-blow-up of $G$}, denoted by $G(t)$, is the graph obtained by replacing each vertex $u$ of $G$ by an independent set $V_u$ with $|V_u|=t$ and replacing each edge $uv$ by all the possible edges between $V_u$ and $V_v$.

            Let $G$ be a bipartite graph with parts $X$ and $Y$, and let $s, t$ be positive integers.
            The {\it $(s, t)$-blow-up of $G$}, denoted by $G(s, t)$, is the bipartite graph obtained by replacing each vertex $u$ of $X$ by an independent set $V_u$ with $|V_u|=s$, and each vertex $v$ of $Y$ by an independent set $V_v$ with $|V_v|=t$, and replacing each edge $uv$ by all the possible edges between $V_u$ and $V_v$.
        \end{definition}

        Now we can give the construction of a $B_t$-free graph.
        Let $n$ be a sufficiently large integer and $t$ be a positive integer.
        Denote $\floor{(t+1)/2}$ by $s$ for convenience.

        By the Tur\'{a}n number of $C_4$, there is a $C_4$-free graph, say $G_1$, of order $n/s$ with
        \begin{equation*}
            e(G_1) = (1 + o(1)) \frac{1}{2} \left(\frac{n}{s}\right)^{\frac{3}{2}}
        \end{equation*}
        edges.
        Let $G = G_1(s)$ be the $s$-blow-up of $G_1$.
        Note that $G$ is of order $n$ and every edge in $G_1$ corresponds to a complete bipartite graph $K_{s, s}$ in $G$ and hence
        \begin{equation*}
            e(G) = s^2 \cdot e(G_1) = (1 + o(1))\frac{\sqrt{\floor{(t+1)/2}}}{2} n^{\frac{3}{2}}.
        \end{equation*}

        \begin{claimthm}\label{cla: 201}
            The graph $G$ is $B_t$-free.
        \end{claimthm}
        \begin{claimproof}
            Suppose to the contrary that $G$ contains a copy of $B_t$.
        Let the $t$ copies of $C_4$ in this $B_t$ be $u v x_i y_i$, where $i = 1, \dots, t$.
        Recall that every vertex $x$ in $G_1$ corresponds to an independent set $V_x$ with size $s$ in $G$.
        Assume that $u \in V_x$ for some $x$ and $v \in V_y$ for some $y$.
        For any $j \in [t]$, we claim that $x_j \in V_x$ or $y_j \in V_y$.
        Indeed, if not, then the $x_j \in V_z$ for some $z$ and $y_j \in V_w$ for some $w$.
        Notice that $w \neq z$, since each vertex in $G_1$ corresponds to an independent set in $G$.
        It follows from the definition of blow-up immediately that $x y z w$ is a copy of $C_4$ in graph $G_1$, contradicting the choice of $G_1$. This shows that $x_j \in V_x$ or $y_j \in V_y$. Observe that $u, x_1, \dots, x_t$ are pairwise distinct and $v, y_1, \dots, y_t$ are pairwise distinct. Hence, we have
        \begin{equation*}
            t+2\le\left|\left(\bigcup_{j=1}^t \{u, v, x_j, y_j\}\right)\cap(V_x\cup V_y)\right|\le |V_x\cup V_y|=2s=2\floor{\frac{t+1}{2}}\le t+1,
        \end{equation*}
        implying a contradiction.
        \end{claimproof}
        
        For $B_{2s}$, we can obtain a better lower bound by using random construction.
        Fix $p \in [0, 1]$.
        Let $m$ be an integer.
        Let $G_0$ be a $C_4$-free graph with $\frac{1}{2} (1+o(1))m^{\frac{3}{2}}$ edges.
        Choose a $pm$-set $A$ uniformly among all $pm$-subsets.
        If $x \notin A$, replace $x$ by $s$ copies.
        If $x \in A$, replace $x$ by $s+1$ copies.
        For each edge $xy$ in $G_0$,
        if $x, y \notin A$, replace it by $K_{s,s}$.
        If exactly one of $x, y$ in $A$, replace it by $K_{s,s+1}$.
        If $x, y \in A$, replace it by $K_{s+1,s+1}$ minus a perfect matching.
        Denote the final graph by $G_1$.
        As in Claim~\ref{cla: 201}, it can be shown that $G_1$ is $B_{2s}$-free. 
        The number of vertices in $G_1$ is
        \begin{equation*}
            v(G_1) = (m-pm)s + pm \cdot (s+1) = sm + pm.
        \end{equation*}
        The number of edges in $G_1$ is 
        \begin{equation*}
            e(G_1) = s^2 \cdot (e(G_0) - \sum_{e \in E(G_0)} 1_{e \cap A \neq \emptyset}) + s(s+1) \cdot \sum_{e \in E(G_0)} 1_{e \cap A \neq \emptyset} = s^2 \cdot e(G_0) + s \sum_{e \in E(G_0)} 1_{e \cap A \neq \emptyset}.
        \end{equation*}
        The probability that an edge in $G_0$ intersects $A$ is $1 - \frac{\binom{m-2}{pm}}{\binom{m}{pm}} = 2p - p^2 +o(1)$.
        Thus
        \begin{equation*}
            E(e(G_1)) = s^2 \cdot e(G_0) + s \cdot (2p - p^2 +o(1)) e(G_0).
        \end{equation*}
        Using $e(G_0) = \frac{1}{2}(1+o(1))m^{\frac{3}{2}}$, we have that there exists a $G_1$ on $(s+p)m$ vertices and at least $(\frac{1}{2} \frac{s^2+s(2p-p^2)}{(s+p)^{\frac{3}{2}}})(1+o(1))v(G_1)^{\frac{3}{2}}$ edges, which gives 
        \begin{equation*}
            \ex(n,B_{2s}) \geq \max_{p \in [0,1]} \left(\frac{1}{2} \frac{s^2+s(2p-p^2)}{(s+p)^{\frac{3}{2}}} \right)(1+o(1))n^{\frac{3}{2}}.
        \end{equation*}
        Let $f(p) = \frac{1}{2} \frac{s^2+s(2p-p^2)}{(s+p)^{\frac{3}{2}}}$.
        The $f(p)$ achieves maximum when $p = \sqrt{4s^2+5s+1} - 2s - 1$.
        Thus
        \begin{equation}
            \ex(n, B_{2s}) \geq \frac{2s(s+1)(\sqrt{4s^2+5s+1}-2s-1)}{(\sqrt{4s^2+5s+1}-s-1)^{\frac{3}{2}}}(1+o(1))n^{\frac{3}{2}}.
        \end{equation}

        For the bipartite $B_t$-free graph, the construction is similar to the construction of a $B_t$-free graph.
        Let $a = \floor{\frac{t+1}{2}}$ and $b = \ceil{\frac{t+1}{2}}$.
        Let $G_1$ be the bipartite $C_4$-free graph with two parts each of size $\floor{n/(t+1)}$ and the maximum number of edges.
        Let $G$ be the $(a, b)$-blow-up of $G_1$.
        If needed, add isolated vertices at the end to reach exactly $n$ vertices.
        By the bipartite Tur\'{a}n number of $C_4$, if $t$ is odd, the graph $G$ has $e(G) = \frac{\sqrt{\floor{(t+1)/2}}}{2\sqrt{2}}(1 + o(1)) n^{\frac{3}{2}}$.
        If $t$ is even, the graph $G$ has $e(G) = \frac{1}{4} \frac{t(t+2)}{(t+1)^{\frac{3}{2}}}(1+o(1)) n^{\frac{3}{2}}$.
        Similar to Claim~\ref{cla: 201}, it can be shown that these bipartite graphs are $B_t$-free.

        \subsection{Useful lemmas}\label{subsec: useful lemmas}

        In the rest of the paper, we focus on the upper bound. 
        Our strategy is to double-count the number of copies of $C_4$ that have two opposite vertices contained in a good set of size~$t$.
        The following lemmas give the lower bound of the number of copies of $C_4$ that have two opposite vertices contained in a good set of size $t$.


        \begin{lemma}\label{lem: lower bound of number of C4 for Bt}
            Let $t$ be a positive integer and $G$ be a $B_t$-free graph on $n$ vertices. When $n$ is sufficiently large, we have
            \begin{equation*}
                \sum_{\substack{S \subseteq V(G), |S| = t \\ S \text{ is good}}} d(S)(d(S) - 1) \geq \sum_{u \in V(G)} d(u)^2 - 2(4t-1) e(G) - (t-1) n^2.
            \end{equation*}
        \end{lemma}

        \begin{proof}
            Let $G$ be a $B_t$-free graph.
            A \emph{spider} is a tree with at most one vertex with degree more than $2$, called the \emph{center}.
            A leg of a spider is a path from the center to a vertex of degree $1$.

            \begin{claimthm}\label{cla: 211}
                For any $v \in V(G)$, $G[N(v)]$ does not contain a spider with $t$ legs of length $2$.
            \end{claimthm}
            \begin{claimproof}
                Suppose not. Then there exists a spider with center $u$ and $t$ legs $u w_1 w_1'$, $u w_2 w_2'$, $\dots$, $u w_t w_t'$.
                Then there is a $B_t$ formed by $u w_1 w_1' v$, $u w_2 w_2' v$, $\dots$, $u w_t w_t' v$ sharing one common edge $uv$, a contradiction.
            \end{claimproof}
            A known result for spiders is: If $G$ is a graph on $n$ vertices with $e(G) > \frac{k-1}{2} n$, then $G$ contains every $k$-edge spider that has no leg of length more than $4$ (see \cite{FanThe2007, WozniakOn1996}).
            Combining with Claim \ref{cla: 211}, we have that the number of edges in $G[N(v)]$ is at most $\frac{2t - 1}{2} d(v)$.


            
            Define an auxiliary hypergraph $\mathcal{H}_v$.
            The vertex set of $\mathcal{H}_v$ is $V(G) \setminus \{v\}$, and the hyperedge set is $\{h: h = (N(v) \cap N(w)) \cup \{w\}, w \in V(G) \setminus N[v]\}$.
            By the definition of $\mathcal{H}_v$, each hyperedge of size $s + 1$ corresponds to $s$ edges incident to vertices in $N(v)$.
            Then we have
            \begin{equation}\label{eq: sum of dv}
                \sum_{u \in N(v)} d(u) = d(v) + 2e(G[N(v)]) + \sum_{h \in E(\mathcal{H}_v)} (|h| - 1).
            \end{equation}
            where the first item on the right is the number of edges between $v$ and $N(v)$, the second item is the number of edges in $G[N(v)]$, and the third item is the number of edges between $N(v)$ and $V(G) \setminus N[v]$.

            Let $\mathcal{H}_{v, 1}$ denote the subhypergraph of $\mathcal{H}_v$ consisting of hyperedges of size at most $t$, $\mathcal{H}_{v, 2}$ denote the subhypergraph of $\mathcal{H}_v$ consisting of hyperedges of size $t+1$, and $\mathcal{H}_{v, 3}$ denote the subhypergraph of $\mathcal{H}_v$ consisting of hyperedges of size at least $t+2$.
            Thus $\sum_{h \in E(\mathcal{H}_v)} (|h| - 1) = \sum_{i = 1}^{3} \sum_{h \in E(\mathcal{H}_{v, i})} (|h| - 1)$.

            \begin{claimthm}\label{cla: 212}
                For each vertex $u \in N(v)$, it is contained in at most $t-1$ hyperedges of $\mathcal{H}_{v, 3}$.
            \end{claimthm}
            \begin{claimproof}
                Suppose not. Then there are at least $t$ hyperedges of $\mathcal{H}_{v, 3}$, namely $h_1, h_2, \dots, h_t$, containing $u$.
                Let $w_i$ be the vertex in $h_i$ but not in $N(v)$.
                Since the size of $h_i$ is at least $t+2$, we can choose vertices $u_i \neq w_i$ from $h_i$ greedily such that $u_i \neq u_j$ for $1 \leq i, j \leq t$.
                Then there are $t$ copies of $C_4$, namely $v u w_1 u_1$, $v u w_2 u_2$, \dots, $v u w_t u_t$, which share a common edge $vu$ and form a $B_t$, a contradiction.
            \end{claimproof}
            Thus we have 
            \begin{equation}\label{eq: sum of edge size}
                \sum_{h \in E(\mathcal{H}_{v, 3})} (|h| - 1) = \sum_{u \in N(v)} |\{h \in E(\mathcal{H}_{v, 3}): u \in h\}| \leq (t-1) d(v).
            \end{equation}

            For each hyperedge $h$ in $\mathcal{H}_{v, 1} \cup \mathcal{H}_{v, 2}$, the size of $h \cap N(v)$ is at most $t$.

            \begin{claimthm}\label{cla: 213}
                Let $S \subseteq N(v)$ be a set of size $t$.
                If $S$ is contained in at least $t$ hyperedges of $\mathcal{H}_{v, 2}$, then no vertex in $S$ belongs to any other hyperedge $g \in \mathcal{H}_{v, 2}$ with $|g \cap S| \leq t-1$ besides these $t$ hyperedges.
            \end{claimthm}
            \begin{claimproof}
                Suppose not. Let $h_1, h_2, \dots, h_t$ be $t$ hyperedges containing $S = \{u_1, \dots, u_t\}$.
                Assume that $u_1$ is also contained in $h_{t+1}$, $|h_{t+1} \cap S| \leq t-1$.
                Let $w_i = h_i \setminus S$, $u_i \in h_i$ such that $u_i \neq u_j$ for $i \neq j$.
                Then there are $t$ copies of $C_4$ $v u_1 w_2 u_2$, $v u_1 w_3 u_3$, $\dots$, $v u_1 w_t u_t$ and $v u_1 w_{t+1} u_{t+1}$ which form a $B_t$, a contradiction.
            \end{claimproof}

            \begin{claimthm}\label{cla: 214}
                Let $S \subseteq N(v)$ be a set of size $t$.
                If $S$ is contained in at least $1$ and at most $t-1$ hyperedges of $\mathcal{H}_{v, 2}$, then each vertex in $S$ is contained in at most $t-1$ hyperedges of $\mathcal{H}_{v, 2}$.
            \end{claimthm}
            \begin{claimproof}
                Suppose not. Let $h_1, h_2, \dots, h_{t_1}$ be $t_1$ hyperedges containing $S = \{u_1, u_2, \dots, u_t\}$, where $1 \leq t_1 \leq t-1$.
                Assume that $u_1$ is also contained in $h_{t_1 + 1}, \dots, h_t$.
                Let $w_i = h_i \setminus N(v)$.

                Consider the bipartite graph $G[W, N(v) \setminus \{u_1\}]$ induced by $W$ and $N(v) \setminus \{u_1\}$, where $W = \{w_1, w_2, \dots, w_t\}$.
                Each vertex in $W$ has $t - 1$ neighbors in $N(v) \setminus \{u_1\}$.
                Let $W' \subseteq W$.
                If $|W'| \leq t-1$, then there are at least $t-1$ vertices adjacent to some vertices in $W'$.
                If $|W'| = t$, since there are at least $t-1$ vertices adjacent to $\{w_1, w_2, \dots, w_{t_1}\}$, the set $S$ is contained in at least $1$ and at most $t-1$ hyperedges of $\mathcal{H}_{v, 2}$.
                Moreover, $h_{t_1+1}, \dots, h_t$ are different from $h_1, \dots, h_{t_1}$. 
                Therefore, there are at least $t$ vertices adjacent to some vertices in $W'$.
                Thus, by Hall's theorem \cite{HallOn1935}, there exist $\{x_1, x_2, \dots, x_t\} \subseteq N(v)$ such that $x_i \in h_i$, $1 \leq i \leq t$.
                Then there are $t$ copies of $C_4$ $v u_1 w_1 x_1$, $v u_1 w_2 x_2$, $\dots$, $v u_1 w_{t-1} x_{t-1}$ and $v u_1 w_t x_t$ which form a $B_t$, a contradiction.
            \end{claimproof}
            \noindent We set 
            $$E_2 = \{h \in E(\mathcal{H}_{v, 2}): 1 \leq |\{g \in \mathcal{H}_{v, 2}: h \cap N(v) \subseteq g\}| \leq t-1\}.$$
            This means $\sum_{h \in E_2} (|h| - 1) \leq (t-1)d(v).$
            The $d(S) - 1$ is the number of common neighbors of $S$ besides $v$.
            Suppose that there are $x \leq n$ hyperedges $h$ of $\mathcal{H}_{v, 2}$ in $E(\mathcal{H}_{v, 2}) \setminus E_2$.
            Thus $\sum_{h \in E(\mathcal{H}_{v, 2}) \setminus E_2} (|h| - 1)  \leq  \sum_{\substack{S \subseteq N(v), |S| = t, S \text{ is good}}} (d(S) - 1) + (t-1)x$.
            There are at most $n-x$ hyperedges in $\mathcal{H}_{v, 1}$, so $\sum_{h \in E(\mathcal{H}_{v, 1})} (|h| - 1) \leq (t-1)(n - x)$.
            Thus 
            \begin{equation*}
                \sum_{i = 1}^{2} \sum_{h \in E(\mathcal{H}_{v, i})} (|h| - 1) \leq (t-1)d(v) + (t-1)n + \sum_{\substack{S \subseteq N(v), |S| = t, \\ S \text{ is good}}} (d(S) - 1).
            \end{equation*}
            
            Combining with Equation (\ref{eq: sum of dv}) and Equation (\ref{eq: sum of edge size}) we have
            \begin{equation*}
                \begin{aligned}
                    \sum_{\substack{S \subseteq N(v), |S| = t \\ S \text{ is good}}} (d(S) - 1)
                    \geq& \sum_{u \in N(v)} d(u) - 2t d(v) - d(v) - (t-1) d(v) - (t-1) d(v) - (t-1) n \\
                    =& \sum_{u \in N(v)} d(u) - (4t-1) d(v) - (t-1)n.
                \end{aligned}
            \end{equation*}
            Summing over all $v \in V(G)$, we have
            \begin{equation*}
                \begin{aligned}
                    & \sum_{v \in V} \sum_{\substack{S \subseteq N(v), |S| = t \\ S \text{ is good}}} (d(S) - 1) \\
                    \geq& \sum_{v \in V} \left(\sum_{u \in N(v)} d(u) - (4t-1) d(v) - (t-1)n \right) \\
                    =& \sum_{u \in V} d(u)^2 - 2(4t-1) e(G) - (t-1) n^2.
                \end{aligned}
            \end{equation*}
            On the other hand, each set $S$ of size $t$ is contained in $d(S)$ neighborhoods. So
            \begin{equation*}
                \sum_{\substack{S \subseteq V, |S| = t \\ S \text{ is good}}} d(S)(d(S) - 1) = \sum_{v \in V} \sum_{\substack{S \subseteq N(v), |S| = t \\ S \text{ is good}}} (d(S) - 1).
            \end{equation*}
            Thus, we have
            \begin{equation*}
                \sum_{\substack{S \subseteq V, |S| = t \\ S \text{ is good}}} d(S)(d(S) - 1) \geq \sum_{u \in V} d(u)^2 - 2(4t-1) e(G) - (t-1) n^2.\qedhere
            \end{equation*} 
        \end{proof}

        For the $B_t$-free bipartite graph, one can obtain a conclusion similar to Lemma \ref{lem: lower bound of number of C4 for Bt}. 
        The proof is very similar, differing only in some computational details.

        \begin{lemma}\label{lem: lower bound of number of C4 for Bt on bip}
            Let $t$ be a positive integer and $G$ be a $B_t$-free bipartite graph with parts $V_1$ and $V_2$, and $|V_i| = n_i$.
            When $n$ is sufficiently large, we have
            \begin{equation*}
                N_i(C_4, G) \geq \frac{1}{2} \binom{t}{2} \left(\sum_{u \in V_i} d(u)^2 - (2t-1) e(G) - (t-1) n_j^2 \right),
            \end{equation*}
            where $i \neq j$, $i, j \in \{1, 2\}$.
        \end{lemma}

        \begin{proof}
            For a vertex $v \in V_1$, let $\mathcal{H}_v$ be the auxiliary hypergraph defined in the proof of Lemma \ref{lem: lower bound of number of C4 for Bt}.
            Since $G$ is a bipartite graph, we have
            \begin{equation}\label{eq: sum of dv bip}
                \sum_{u \in N(v)} d(u) = d(v) + \sum_{h \in E(\mathcal{H}_v)} (|h| - 1).
            \end{equation}

            Let $\mathcal{H}_{v, i}$ ($1 \leq i \leq 3$) be the subhypergraph of $\mathcal{H}_v$ as defined in the proof of Lemma \ref{lem: lower bound of number of C4 for Bt}.

            As in the proof of Lemma \ref{lem: lower bound of number of C4 for Bt}, we have 
            \begin{equation}\label{eq: lem2.2 1}
                \sum_{h \in E(\mathcal{H}_{v, 3})} (|h| - 1) = \sum_{u \in N(v)} |\{h \in E(\mathcal{H}_{v, 3}): u \in h\}| \leq (t-1) d(v).
            \end{equation}
            Notice that $v \in V_1$,
            by the construction of $\mathcal{H}_v$, each hyperedge in $\mathcal{H}_v$ consists of a vertex in $V_1$ together with its neighbors.
            Thus, there are at most $n_1$ hyperedges in $\mathcal{H}_v$.
            Suppose that there are $x \leq n_1$ hyperedges $h$ of $\mathcal{H}_{v, 2}$ in $E(\mathcal{H}_{v, 2}) \setminus E_2$.
            Thus $\sum_{h \in E(\mathcal{H}_{v, 2}) \setminus E_2} (|h| - 1)  \leq  \sum_{\substack{S \subseteq N(v), |S| = t, S \text{ is good}}} (d(S) - 1) + (t-1)x$.
            There are at most $n_1-x$ hyperedges in $\mathcal{H}_{v, 1}$, so $\sum_{h \in E(\mathcal{H}_{v, 1})} (|h| - 1) \leq (t-1)(n_1 - x)$.
            As in the proof of Lemma \ref{lem: lower bound of number of C4 for Bt}, we have
            \begin{equation}\label{eq: lem2.2 2}
                \sum_{i = 1}^{2} \sum_{h \in E(\mathcal{H}_{v, i})} (|h| - 1) \leq (t-1)d(v) + (t-1)n_1 + \sum_{\substack{S \subseteq N(v), |S| = t, \\ S \text{ is good}}} (d(S) - 1).
            \end{equation}

            Combining Equation (\ref{eq: lem2.2 1}), (\ref{eq: lem2.2 2}) and (\ref{eq: sum of dv bip}), we have
            \begin{equation*}
                \begin{aligned}
                \sum_{\substack{S \subseteq N(v), |S| = t \\ S \text{ is good}}} (d(S) - 1)
                \geq& \sum_{u \in N(v)} d(u) - d(v) - (t-1) d(v) - (t-1) d(v) - (t-1) n_1 \\
                = & \sum_{u \in N(v)} d(u) - (2t-1) d(v) - (t-1) n_1.
                \end{aligned}
            \end{equation*}
            Summing over all $v \in V_1$, we have
            \begin{equation*}
                \begin{aligned}
                    & \sum_{v \in V_1} \sum_{\substack{S \subseteq N(v), |S| = t \\ S \text{ is good}}} (d(S) - 1) \\
                    \geq& \sum_{v \in V_1} \left(\sum_{u \in N(v)} d(u) - (2t-1) d(v) - (t-1) n_1 \right) \\
                    =& \sum_{u \in V_2} d(u)^2 - (2t-1) e(G) - (t-1) n_1^2.
                \end{aligned}
            \end{equation*}
            On the other hand, each set $S$ of size $t$ can be contained in $d(S)$ neighborhoods. So
            \begin{equation*}
                \sum_{\substack{S \subseteq V_2, |S| = t \\ S \text{ is good}}} d(S)(d(S) - 1) = \sum_{v \in V_1} \sum_{\substack{S \subseteq N(v), |S| = t \\ S \text{ is good}}} (d(S) - 1).
            \end{equation*}
            Thus, we have
            \begin{equation}\label{eq: lower bound of number of C4 1}
                \sum_{\substack{S \subseteq V_2, |S| = t \\ S \text{ is good}}} d(S)(d(S) - 1) \geq \sum_{u \in V_2} d(u)^2 - (2t-1) e(G) - (t-1) n_1^2.
            \end{equation}

            For each set $S \subseteq V_2$ of size $t$, any two common neighbors of $S$ together with two vertices in $S$ form a $C_4$.
            Thus
            \begin{equation}\label{eq: sum of dv 2}
                N_2 (C_4, G) = \sum_{\substack{S \subseteq V_2, |S| = t \\ S \text{ is good}}} \binom{d(S)}{2} \binom{t}{2}.
            \end{equation}

            Combining Equation (\ref{eq: lower bound of number of C4 1}) and Equation (\ref{eq: sum of dv 2}), we have
            \begin{equation*}
                N_2(C_4, G) \geq \frac{1}{2} \binom{t}{2} \left(\sum_{u \in V_2} d(u)^2 - (2t-1) e(G) - (t-1) n_1^2 \right).
            \end{equation*}

            Similarly to the above, we obtain $N_1(C_4, G) \geq \frac{1}{2} \binom{t}{2} (\sum_{u \in V_1} d(u)^2 - (2t-1) e(G) - (t-1) n_2^2)$.
            The proof is complete.
        \end{proof}




        For a $B_t$-free graph $G$, let $\mathcal{C} = \{C_4 = u_1 v_1 u_2 v_2: S \text{ is a good set of size } t, \{u_1, u_2\} \subseteq S \text{ or } \{v_1, v_2\} \subseteq S\}$.
        Thus $|\mathcal{C}| = N(C_4, G)$.
        We count the number of copies of $C_4 \in \mathcal{C}$ by first choosing two vertices in a good set, then choosing two common neighbors of this good set.
        A fixed $C_4 = x_1 y_1 x_2 y_2$ cannot be counted twice from the same opposite pair.
        Suppose otherwise. Then there are two distinct good sets $X = \{x_1, x_2, x_3, \dots, x_t\}$ and $X' = \{x_1, x_2, x_3', \dots, x_t'\}$ both containing $x_1$ and $x_2$.
        Let $\{y_1, y_2, y_3, \dots, y_{t+1}\} \subseteq N(X)$ and $u \in X \setminus X'$.
        Then $x_1 y_1 u y_2$, $x_1 y_1 x_2 y_3$, $x_1 y_1 x_3 y_4$, $\dots$, $x_1 y_1 x_t y_{t+1}$ share a common edge $x_1 y_1$ and form a $B_t$.
        So each counted $C_4$ is counted at most once from each opposite pair, hence at most twice in total.
        The following lemma shows that there are a few copies of $C_4 \in \mathcal{C}$ which are counted twice.

        \begin{lemma}\label{lem: better lower bound of Bt}
            Let $t$ be a positive integer and $G$ be a $B_t$-free graph on $n$ vertices. When $n$ is sufficiently large, we have
            \begin{equation*}
                N(C_4, G) + o(n^2) \geq \frac{1}{2} \binom{t}{2} \left(\sum_{u \in V(G)} d(u)^2 - 2(4t-1) e(G) - (t-1) n^2 \right).
            \end{equation*}
        \end{lemma}
        \begin{proof}

        Let $\mathcal{C}_4$ be the collection of copies of $C_4 \in \mathcal{C}$ that are counted twice.
        Suppose $x_1y_1x_2y_2 \in \mathcal{C}_4$,
        which means $x_1,x_2$ together with other $t-2$ vertices $x_3,\dots, x_t$ form a good set, and $y_1,y_2$ together with other $t-2$ vertices $y_3,\dots,y_t$ form a good set.
        Here $\{x_3,\dots,x_t\}$ might intersect with $\{y_3,\dots,y_t\}$.
        Let $X = \{x_1, x_2, \dots, x_t\}$, $Y = \{y_1, y_2, \dots, y_t\}$.


        Choose $t+1$ vertices from $N(X)$ including $y_1$ and $y_2$, and denote them by $Z=\{z_1,\dots,z_{t+1}\}$ where $z_1=y_1,z_2=y_2$.
        Moreover, $N(y_1,y_2)\subseteq X\cup Z$.
        Suppose that there is a vertex $u \in N(y_1,y_2)\setminus(X\cup Z)$. Then $x_1 y_2 u y_1$, $x_1 y_2 x_2 z_3$, $x_1 y_2 x_3 z_4$, $x_1 y_2 x_4 z_5$, $\dots$, $x_1 y_2 x_t z_{t+1}$ share a common edge $x_1 y_2$ and form a $B_t$ (see Figure \ref{fig1-1}). 
        And $N(Y)\subseteq X\cup Z.$

        \begin{claimthm}\label{cla: 311}
            $\{y_3,\dots,y_t\}\subseteq X\cup N(X)$.
        \end{claimthm}
        \begin{claimproof}
        Suppose not. Then there exists $y_j\not\in X\cup N(X)$ for some $j\in [3,t]$. We may assume $j=t$.
        Since $d(y_1,y_2,y_t)\geq t+1$ and $N(y_1,y_2)\subset X\cup Z$, there exist a vertex in $Z$, namely $z_{t+1}\in Z$, and $x_i$ such that $z_{t+1},x_i\in N(y_1,y_2,y_t)$.
        If $3 \leq i \leq t$, namely $x_i = x_t$, then $x_t y_2 z_{t+1} y_t$, $x_t y_2 x_1 y_1$, $x_t y_2 x_2 z_3$, $x_t y_2 x_3 z_4$, $x_t y_2 x_4 z_5$, $\dots$, $x_t y_2 x_{t-1} z_t$ share the edge $x_t y_2$ and form a $B_t$ (see Figure \ref{fig1-2}). 
        If $1 \leq i \leq 2$, namely $x_i = x_2$, then $x_2 y_2 z_{t+1} z_t$, $x_2 y_2 x_1 y_1$, $x_2 y_2 x_3 z_3$, $x_2 y _2 x_4 z_4$, $\dots$, $x_2 y_2 x_t z_t$ share the edge $x_2 y_2$ and form a $B_t$.
        \end{claimproof}

        \begin{figure}[htbp]
            \centering
            \subfloat[The common neighbors of $y_1$ and $y_2$ are contained in $X \cup Z$]{
                \includegraphics[scale=0.18]{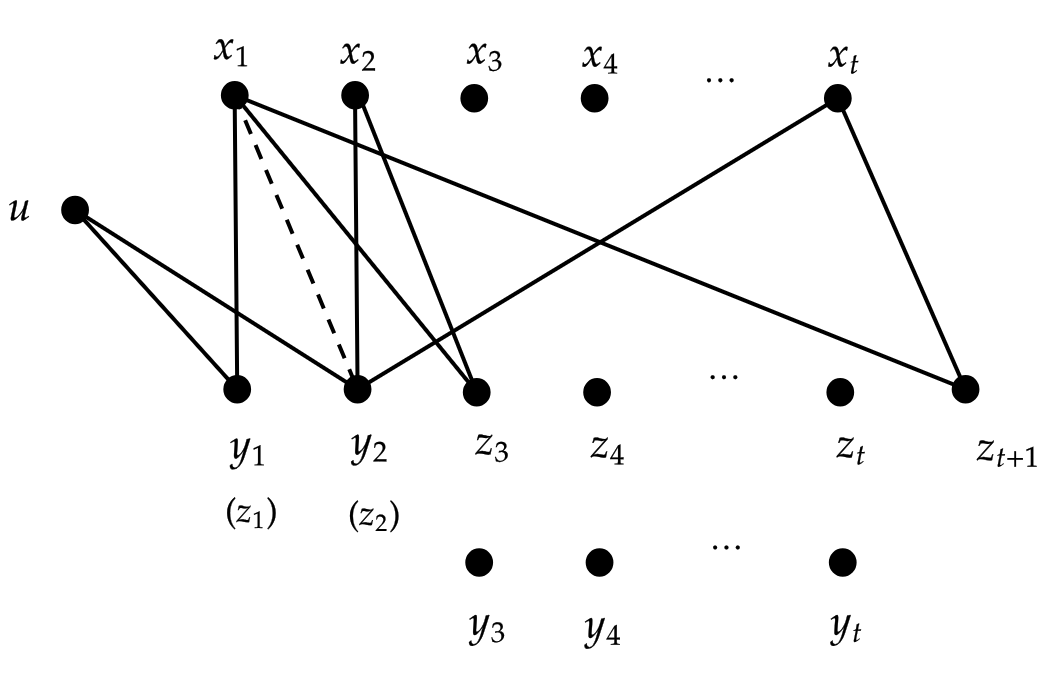} 
                \label{fig1-1}
            }
            \hfill
            \subfloat[$\{y_3, \dots, y_t\}$ is a subset of $X \cup N(X)$]{
                \includegraphics[scale=0.18]{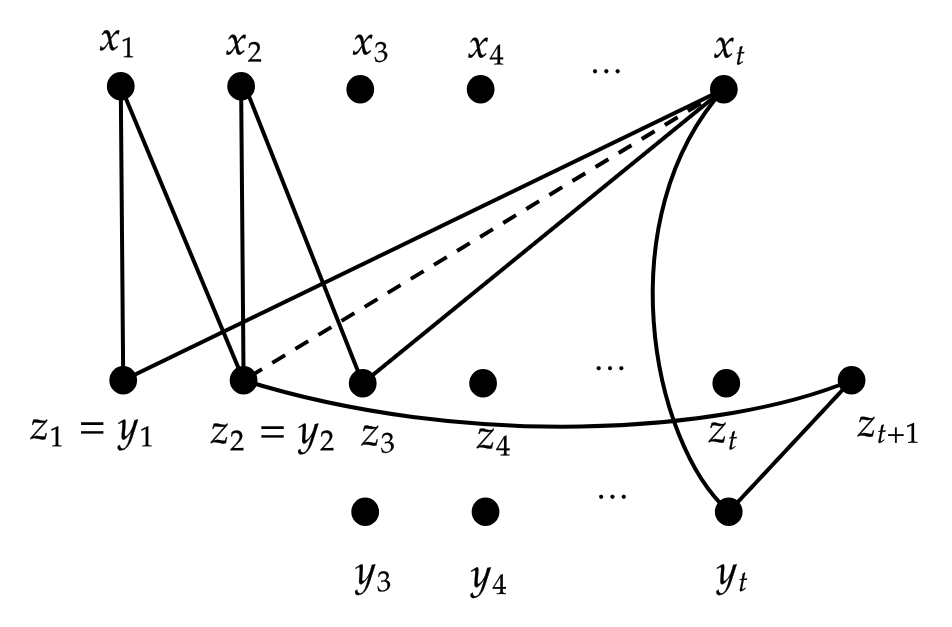} 
                \label{fig1-2}
            }
            \\
            \subfloat[The neighborhood of $X$ is $Z$]{
                \includegraphics[scale=0.18]{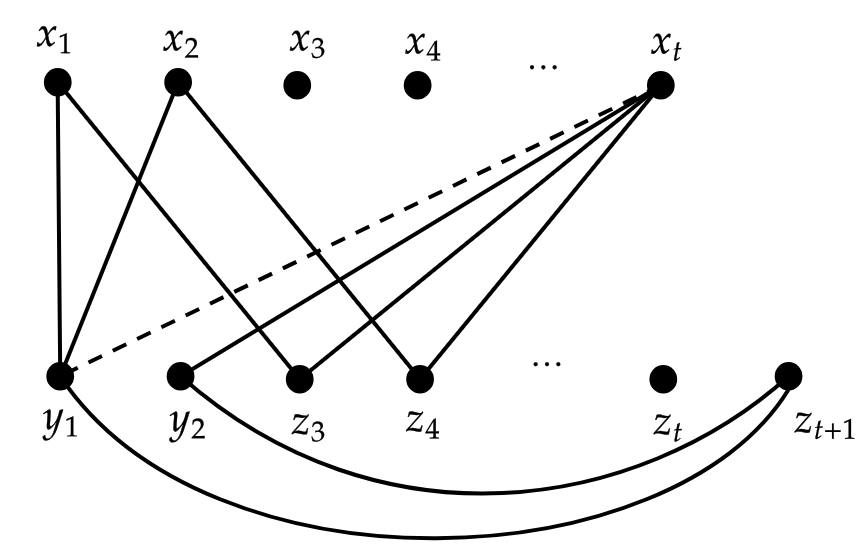} 
                \label{fig1-4}
            }
            \hfill
            \subfloat[For every pair of vertices in a special pair $X$ and $Y$, their common neighbors are contained in $X \cup N(X)$]{
                \includegraphics[scale=0.18]{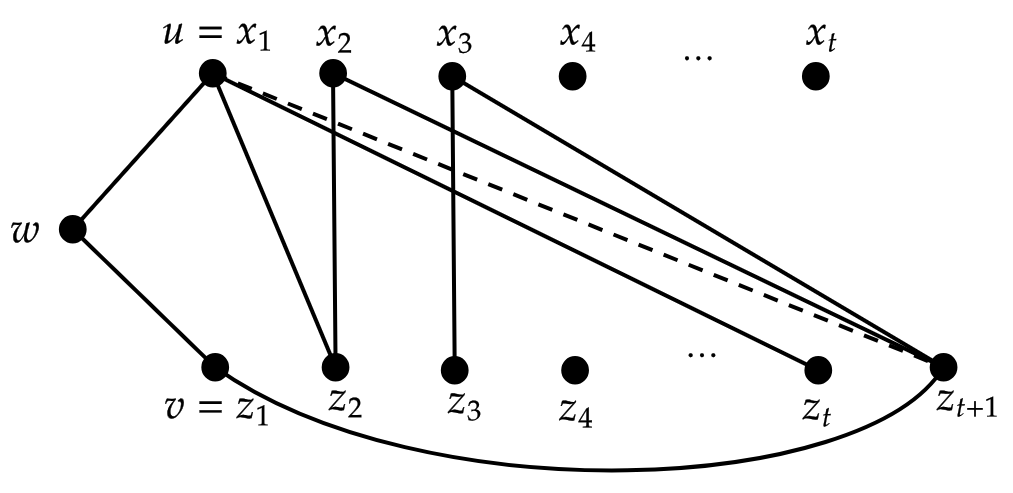} 
                \label{fig1-7}
            }
            \caption{The red edge represents the edge shared by $t$ copies of $C_4$}
        \end{figure}

        Suppose $N(X) \setminus Z \neq \emptyset$ and let $u \in N(X) \setminus Z$.
        Since $d(y_1,y_2)\geq t+1$ and $N(y_1,y_2)\subset X\cup Z$, there exist a vertex in $Z$, namely $z_{t+1}\in Z$, and $x_i$ such that $z_{t+1},x_i\in N(y_1,y_2)$.
        We regard the vertex set $X \cup N(X)$ and the edges between $X$ and $N(X)$ as a complete bipartite graph.
        So these $x_i$s are the same.
        Suppose $x_i = x_t$.
        Then $x_t y_1 z_{t+1} y_2$, $x_t y_1 x_1 z_3$, $x_t y_1 x_2 z_4$, $x_t y_1 x_3 z_5$, $\dots$, $x_t y_1 x_{t-2} z_{t+1}$, $x_t y_1 x_{t-1} u$ share a common edge $x_t y_1$ and form a $B_t$, a contradiction (see Figure \ref{fig1-4}).

        Thus
        \begin{equation*}
            N(X) = Z.
        \end{equation*}

        We call such two good sets $X$ and $Y$ a {\it special pair} and call the $C_4=x_1y_1x_2y_2$ a {\it special $C_4$}.

        \begin{claimthm}\label{cla: 312}
            If two good sets $X$ and $Y$ are a special pair, then for every two vertices $u,v\in X\cup N(X)$, we have $$N(u,v)\subseteq X\cup N(X).$$
        \end{claimthm}
        \begin{claimproof}
            We regard $X\cup N(X)$ and edges between $X$ and $N(X)$ as a complete bipartite graph.
            So these $x_i$s are the same, and these $z_i$ are the same.
            
            If $u,v\in N(X)$. Suppose $u = z_1, v = z_2$ and the common neighbor $w \notin X\cup N(X)$.
            Then $x_1 z_2 w z_1$, $x_1 z_2 x_2 z_3$, $x_1 z_2 x_3 z_4$, $\dots$, $x_1 z_2 x_t z_{t+1}$ share a common edge $x_1 z_2$ and form a $B_t$, a contradiction.
            
            If $u,v\in X$. 
            Suppose that there exists $w\not\in X \cup N(X)$ such that $w\in N(u,v)$.
            Without loss of generality, we may assume $u=x_1$ and $v=x_2$.
            Since $N(y_1,y_2)\subseteq X\cup N(X)$ and $d(y_1,y_2) \geq t+1$, there is a vertex in $N(X)$ adjacent to $y_1$ and $y_2$, namely $z_{t+1}$.
            Then consider the copies of $C_4$: $x_1 y_2 z_{t+1} y_1$, $x_1 y_2 x_2 w$, $x_1 y_2 x_3 z_3$, $x_1 y_2 x_4 z_4$, $\dots$, $x_1 y_2 x_t z_{t}$. They share the edge $x_1y_2$, which forms a $B_t$, a contradiction. 

            Similarly, if $u\in X$ and $v\in N(X)$, then we can also get a contradiction as in the case when $u,v\in X$. For example, there exists a $B_t$ formed by $x_1 z_{t+1} y_1 w$, $x_1 z_{t+1} x_2 y_2$, $x_1 z_{t+1} x_3 z_3$, $\dots$, $x_1 z_{t+1} x_t z_t$ (see Figure \ref{fig1-7}). 
        \end{claimproof}
        For a special pair $X$ and $Y$, we say $\mathcal{V}$ is a special set if $\mathcal{V}=X\cup N(X)$.
        Notice that $\mathcal{V}=X\cup N(X) = Y\cup N(Y)$, since $Y$ is a good set and $N(Y) \subseteq X\cup N(X).$
        Two different special pairs could define the same special set.

        \begin{claimthm}\label{cla: 313}
            Let $\mathcal{V}_1$ and $\mathcal{V}_2$ be two different special sets. Then $G[\mathcal{V}_1]$ and $G[\mathcal{V}_2]$ share no edge.
        \end{claimthm}
        \begin{claimproof}
        Otherwise, suppose that there exists an edge $uv$ in both $G[\mathcal{V}_1]$ and $G[\mathcal{V}_2]$.
        Let $\mathcal{V}_1$ be supported by $X$ and $Y$.

        Case 1. The edge $uv$ is contained in $X$ or $N(X)$.

        Suppose that $uv$ is contained in $X$.
        By Claim \ref{cla: 312}, the common neighbors of $u$ and $v$ are in $\mathcal{V}_2$, which means $N(X)$ is contained in $\mathcal{V}_2$.
        Again by Claim \ref{cla: 312}, the common neighbors of two vertices of $N(X)$ are contained in $\mathcal{V}_2$, which means $X$ is contained in $\mathcal{V}_2$.
        Then we have $\mathcal{V}_1 = \mathcal{V}_2$, a contradiction.
        If $uv$ is contained in $N(X)$.
        Similarly to the above, we can find a contradiction.

        Case 2. The one endvertex of $uv$ is contained in $X$ and the other is contained in $N(X)$.

        Suppose that $u \in X$ and $v \in N(X)$.
        If $v \in Y$, it is easy to see that $d_{G[\mathcal{V}_1]}(v) \geq t+1$.
        If $v \notin Y$, then $v$ must be the common neighbor of $y_1$ and $y_2$, and $v$ is a common neighbor of $X$.
        Thus, $d_{G[\mathcal{V}_1]}(v) \geq t+1$.
        Since the degree of $v$ is at least $t+1$, there is a vertex $w \in N(X)$ adjacent to $v$.
        By Claim \ref{cla: 312}, $w$ is a common neighbor of $u$ and $v$, and $w \in \mathcal{V}_2$.
        Then $vw$ is an edge and is contained in $G[\mathcal{V}_2]$.
        By Case 1, we can find a contradiction.

        From Cases 1 and 2, it follows that no edge is contained in both $G[\mathcal{V}_1]$ and $G[\mathcal{V}_2]$.
        \end{claimproof}

        \begin{claimthm}\label{cla: 314}
            If two different special $C_4$ share at least one edge, then they are from the same special set.
        \end{claimthm}
        \begin{claimproof}
            Suppose $v_1v_2v_3v_4$ and $u_1u_2u_3u_4$ share an edge $e$. If they are in different special sets $\mathcal{V}_1$ and $\mathcal{V}_2$, then by Claim~\ref{cla: 313}, $G[\mathcal{V}_1]$ and $G[\mathcal{V}_2]$ share no edges, a contradiction.
        \end{claimproof}

        For a special set $\mathcal{V}$, the number of edges in $G[\mathcal{V}]$ is at least $t(t+1)$ and at most $\binom{2t+1}{2}$.
        We construct an auxiliary simple hypergraph $\mathcal{H}$ whose vertex set is $E(G)$, and whose hyperedge set is $\{E(G[\mathcal{V}]): \mathcal{V} \text{ is a special set}\}$.
        By Claim \ref{cla: 313}, every two hyperedges in $\mathcal{H}$ share no vertices.
        Since we are forbidding a bipartite graph in $G$, we have $e(G) = o(n^2)$.
        As a result, the number of special sets is at most $\frac{e(G)}{t(t+1)}=o(n^2)$. 
        And the number of special $C_4$ is at most $e(\mathcal{H})\cdot \binom{2t+1}{4} \cdot 4!=o(n^2)$.
        Thus, only $o(n^2)$ copies of $C_4$ are counted twice.
        For each good set $S$, any two common neighbors and two vertices in $S$ form a $C_4$. So
        \begin{equation*}
            N(C_4, G) + o(n^2) \geq \sum_{\substack{S \subseteq V, |S| = t \\ S \text{ is good}}} \binom{d(S)}{2} \binom{t}{2}.
        \end{equation*}
        Together with Lemma \ref{lem: lower bound of number of C4 for Bt}, we have
        \begin{equation*}
            N(C_4, G) + o(n^2) \geq \frac{1}{2} \binom{t}{2} \left(\sum_{u \in V} d(u)^2 - 2(4t-1) e(G) - (t-1) n^2 \right). \qedhere
        \end{equation*}
        \end{proof}

    \section{Proof of Theorem \ref{upper bound of Bt} and Theorem \ref{thm: upper and lower bound of B2}}\label{sec: Bt}

        In this section, we will give the proofs of Theorem \ref{upper bound of Bt} and Theorem \ref{thm: upper and lower bound of B2}.
        Recall that the statement of Theorem \ref{upper bound of Bt} is $\ex(n, B_t) \leq \frac{\sqrt{t}}{2} (1 + o(1)) n^{\frac{3}{2}}$ for sufficiently large $n$.

        \begin{proof}[Proof of Theorem \ref{upper bound of Bt}]
            The case $t = 1$ is the classical $C_4$ bound.
            So we assume that $t \geq 2$.
            Let $G$ be a $B_t$-free graph.
            Choose an edge $uv$ and construct an auxiliary graph $G_{uv}$.
            The vertex set of $G_{uv}$ is $(N(u) \setminus \{v\}) \cup (N(v) \setminus \{u\})$, and the edge set consists of edges with one vertex in $N(u) \setminus \{v\}$ and the other vertex in $N(v) \setminus \{u\}$.

            \begin{claimthm}\label{cla: 315}
                The graph $G_{uv}$ contains no $t$ independent edges.
            \end{claimthm}
            \begin{claimproof}
                Suppose that there are $t$ independent edges $u_1 v_1$, $u_2 v_2$, $\dots$, $u_t v_t$ in $G_{uv}$.
                Then there are $t$ copies of $C_4$, namely $uv u_1 v_1$, $uv u_2 v_2$, $\dots$, $uv u_t v_t$, which share a common edge $vu$ and form a $B_t$, a contradiction.
            \end{claimproof}
            For a fixed edge $uv$, let $A = N(u) \setminus (N(v) \cup \{v\})$, $B = N(v) \setminus (N(u) \cup \{u\})$ and $L = N(u) \cap N(v)$.
            Hence the number of copies of $C_4$ containing $uv$ is $e(G_{uv}[A, B]) + e(G_{uv}[A, L]) + e(G_{uv}[B, L]) + 2e(G_{uv}[L]) = e(G_{uv}) + e(G_{uv}[L])$.
            By Claim \ref{cla: 315}, both $G_{uv}$ and $G_{uv}[L]$ have matching number at most $t-1$.            
            By the Tur\'{a}n number of matchings (see \cite{ErdosOn1959}), 
            there are at most $(t-1) (d(u) + d(v))$ $C_4$s containing $uv$.
            Recall the definition of $N(C_4, G)$, which is at most the number of copies of $C_4$ in $G$.
            Summing over all edges, we have
            \begin{equation*}
                4 N(C_4, G) \leq \sum_{uv \in E} (t-1) (d(u) + d(v)) \leq (t-1) \sum_{u \in V} d(u)^2,
            \end{equation*}
            the coefficient $4$ of $N(C_4, G)$ arises because each $C_4$ is counted $4$ times.



        Together with Lemma \ref{lem: better lower bound of Bt} and the upper bound $4N(C_4, G) \leq (t-1) \sum_{u \in V} d(u)^2$, we have
        \begin{equation*}
            \frac{t-1}{4} \sum_{u \in V} d(u)^2 + o(n^2) \geq \frac{1}{2} \binom{t}{2} \left(\sum_{u \in V} d(u)^2 - 2(4t-1) e(G) - (t-1) n^2 \right),
        \end{equation*}
        i.e.,
        \begin{equation*}
            e(G)^2 - \frac{t(4t-1)n}{2(t-1)}e(G) - \frac{t}{4} n^3 - o(n^3) \leq 0.
        \end{equation*}
        Solving this inequality, we have $e(G) \leq \frac{\sqrt{t}}{2} (1+o(1)) n^{\frac{3}{2}}$.
        \end{proof}

        Now we give a better upper bound for $\ex(n, B_2)$.
        Recall that the statement of Theorem~\ref{thm: upper and lower bound of B2} is $\ex(n,B_2) \leq \frac{2}{\sqrt{11}} (1+o(1)) n^{\frac{3}{2}}$ for sufficiently large $n$.

        \begin{lemma}\label{lem: with large min degree}
            Let $G$ be a $B_2$-free graph on $n$ vertices with minimum degree $\delta \geq c n^{\frac{1}{2}}$. Then we have
            \begin{equation*}
                e(G) \leq \frac{-c + \sqrt{c^2 + 32}}{8} (1 + o(1)) n^{\frac{3}{2}}. 
            \end{equation*}
        \end{lemma}
        \begin{proof}
            Let $G$ be a $B_2$-free graph with $\delta \geq c n^{\frac{1}{2}}$.
            Let $G_{uv}$ be the graph consisting of edges with one vertex in $N(u) \setminus \{v\}$ and another vertex in $N(v) \setminus \{u\}$.
            It is easy to see that each edge in $G_{uv}$ together with $uv$ forms a $C_4$.

            \begin{claimthm} \label{twoind}
             The graph $G_{uv}$ contains no two independent edges.
             \end{claimthm}
             
            \begin{claimproof}
                Suppose that there are $2$ independent edges $u_1 v_1$ and $u_2 v_2$ in $G_{uv}$.
                Then, there are two copies of $C_4$, $uv u_1 v_1$ and $uv u_2 v_2$, which share a common edge $uv$ and form a $B_2$, a contradiction.
            \end{claimproof}
            By Claim~\ref{twoind}, we can obtain that $G_{uv}$ is a star or a triangle.

            Let $L_{v \setminus u} = N(v) \setminus N[u]$, $L_{uv} = N(u) \cap N(v)$ and $L_{u \setminus v} = N(u) \setminus N[v]$.
            Let $L_{v \setminus u}' = L_{v \setminus u} \cap V(G_{uv})$, $L_{uv}' = L_{uv} \cap V(G_{uv})$, and $L_{u \setminus v}' = L_{u \setminus v} \cap V(G_{uv})$.
            We set 
            $$N_1 (u) = \{w \in N(u): |L_{w \setminus u}'| \leq \frac{n^{\frac{1}{2}}}{\log n} \text{ or } |L_{u \setminus w}'| \leq \frac{n^{\frac{1}{2}}}{\log n}\},$$ 
            and $N_2 (u) = N(u) \setminus N_1(u)$.
            By Claim~\ref{twoind} and the definition of $N_2(u)$, for each vertex $v \in N_2(u)$, the graph $G_{uv}$ is a star.
            We call the vertex with degree at least $2$ in $G_{uv}$ a center.
            The center must lie in $L_{uv}'$.

            \begin{claimthm}\label{cla: 322}
                For each vertex $v \in N_2(u)$, $L_{u \setminus v}' \subseteq N_1(u)$.
            \end{claimthm}

            \begin{claimproof}
                Let $v \in N_2(u)$.
                Suppose that $w \in L_{u \setminus v}'$ and $w \notin N_1(u)$, which means $|L_{u \setminus w}'| \geq \frac{n^{\frac{1}{2}}}{\log n}$ and $|L_{w \setminus u}'| \geq \frac{n^{\frac{1}{2}}}{\log n}$.
                Let vertex $y$ be the center of star $G_{uw}$, and $z$ be the center of star $G_{uv}$.

                We claim that $y$ must be $z$ or $v$.
                Suppose otherwise. Then $y \notin \{z, v\}$.
                Then $u w x_1 y$ and $u w z v$ form a $B_2$, where $x_1 \in L_{w \setminus u}'$ (see Figure \ref{fig2-1}).

                Suppose that $y = z$.
                Then $u y x_1 w$ and $u y x_2 v$ form a $B_2$, where $x_1 \in L_{w \setminus u}'$ and $x_2 \in L_{v \setminus u}'$ (see Figure \ref{fig2-2}).

                Suppose that $y = v$.
                Then $u w z x_1$ and $u w x_2 v$ form a $B_2$, where $x_1 \in L_{u \setminus v}'$ and $x_2 \in L_{w \setminus u}'$. 

                \begin{figure}[t]
                    \centering
                    \subfloat[If $w \in L_{u \setminus v}'$ and $z \neq y$, then there exists a $B_2$]{
                        \includegraphics[scale=0.5]{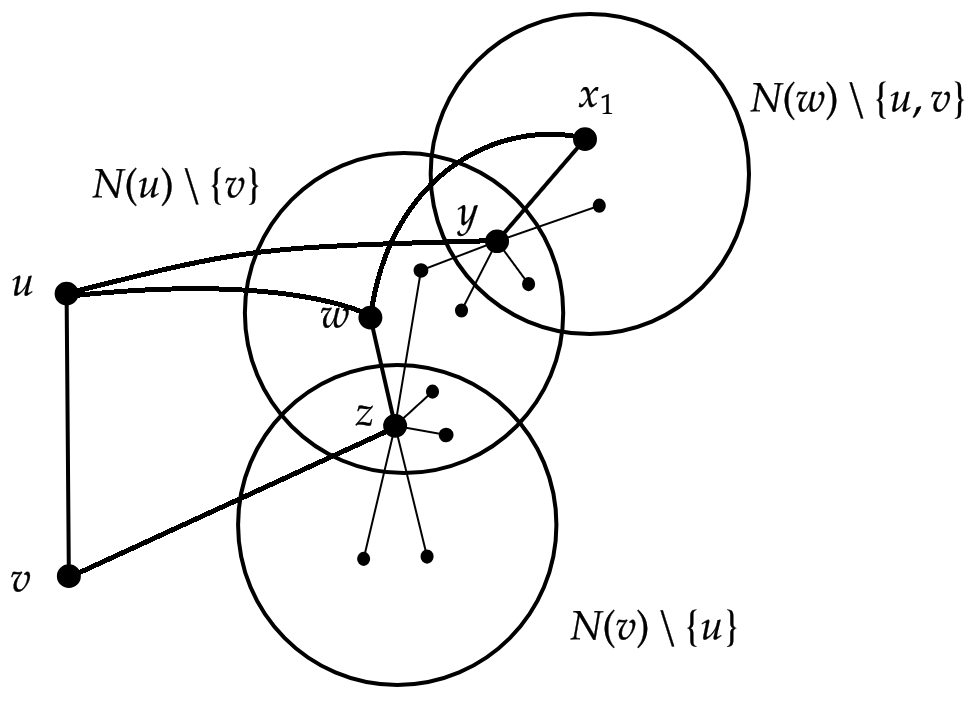} 
                        \label{fig2-1}
                    }
                    \hfill
                    \subfloat[If $w \in L_{u \setminus v}'$ and $z = y$, then there exists a $B_2$]{
                        \includegraphics[scale=0.2]{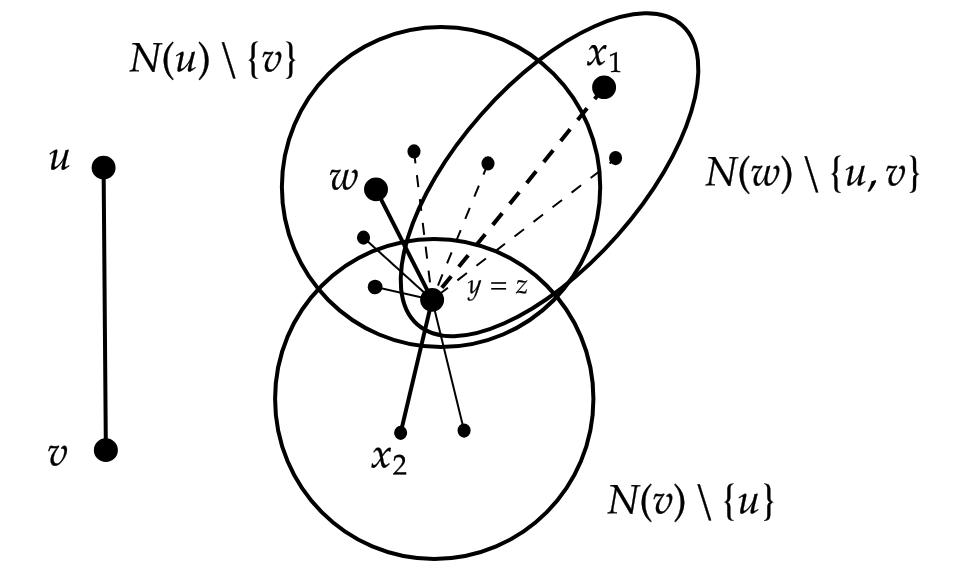} 
                        \label{fig2-2}
                    }
                    \caption{For each vertex $v \in N_2(u)$, if $w \in L_{u \setminus v}'$ and $w \notin N_1(u)$, then we can find a $B_2$}
                \end{figure}

                So, for any vertex $w \in L_{u \setminus v}'$, $w \in N_1(u)$.
            \end{claimproof}

            \begin{claimthm}\label{cla: 323}
                For any two vertices $v_1, v_2 \in N_2(u)$, $L_{u \setminus v_1}' \cap L_{u \setminus v_2}' = \emptyset$.
            \end{claimthm}
            \begin{claimproof}
                Suppose that there is a vertex $w \in L_{u \setminus v_1}' \cap L_{u \setminus v_2}'$.

                Let vertex $y$ be the center of star $G_{u v_1}$, and $z$ be the center of star $G_{u v_2}$.

                Suppose that $y \neq z$.
                If $v_2 \neq y$ and $v_1 \neq z$, then $u v_1 y w$ and $u v_2 z w$ form a $B_2$ (see Figure \ref{fig3-1}).
                If $v_2 = y$ and $v_1 \neq z$, then $u v_1 y w$ and $u w z x_1$ form a $B_2$, where $x_1 \in L_{u \setminus v_2}'$. 
                If $v_2 \neq y$ and $v_1 = z$, then $u v_2 x_1 z$ and $u z w y$ form a $B_2$, where $x_1 \in L_{v_2 \setminus u}'$. 
                If $v_2 = y$ and $v_1 = z$, then $u w v_1 x_1$ and $u w v_2 x_2$ form a $B_2$, where $x_1 \in L_{u \setminus v_2}'$ and $x_2 \in L_{u \setminus v_1}'$ (see Figure \ref{fig3-4}).

                If $y = z$, then $u y x_1 v_1$ and $u y x_2 v_2$ share the edge $uy$ and form a $B_2$, where $x_1 \in L_{v_1 \setminus u}'$ and $x_2 \in L_{v_2 \setminus u}'$.

                \begin{figure}[t]
                    \centering
                    \subfloat[If $w$ is neither the center of $G_{uv_1}$ nor that of $G_{uv_2}$ and $v_2 \neq y$ and $v_1 \neq z$, then there exists a $B_2$]{
                        \includegraphics[scale=0.5]{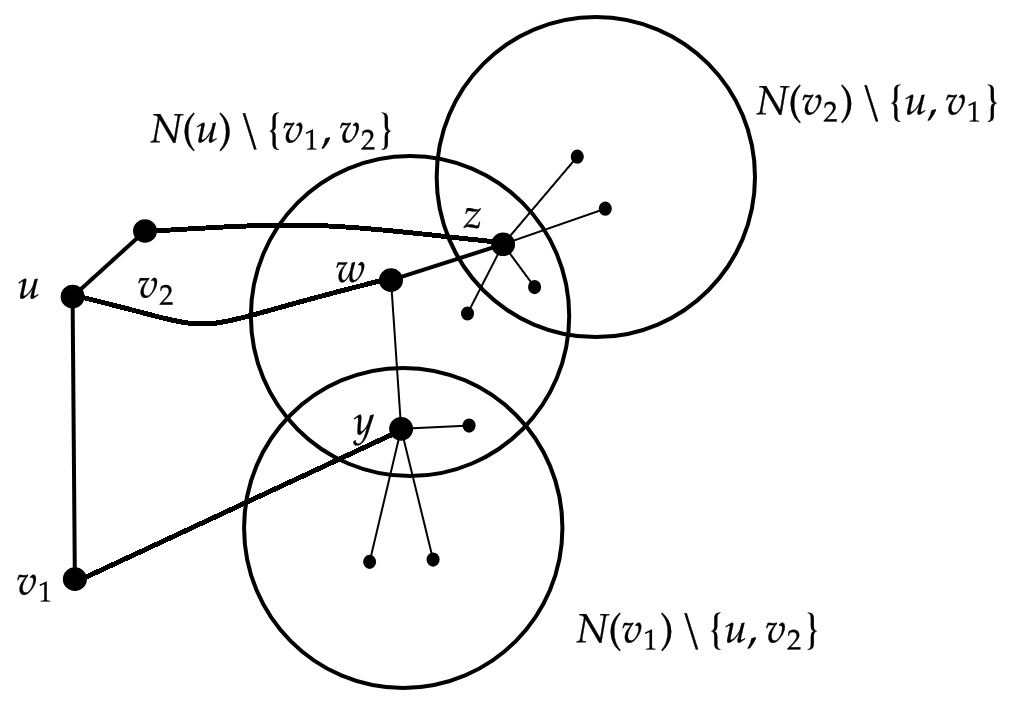} 
                        \label{fig3-1}
                    }
                    \hfill
                    \subfloat[If $w$ is neither the center of $G_{uv_1}$ nor that of $G_{uv_2}$ and $v_2 = y$ and $v_1 = z$, then there exists a $B_2$]{
                        \includegraphics[scale=0.5]{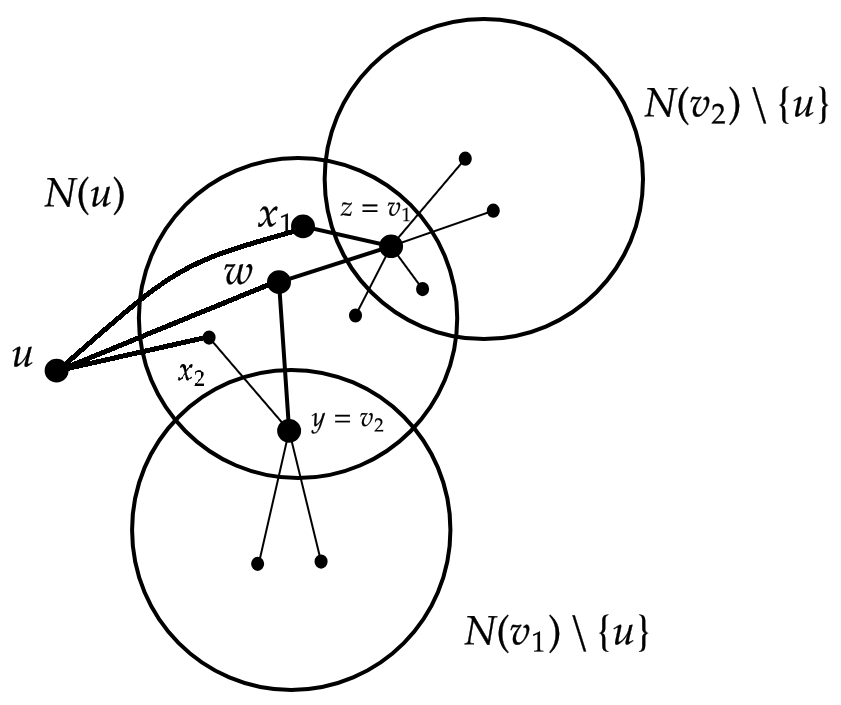} 
                        \label{fig3-4}
                    }
                    \caption{For any two vertices $v_1, v_2 \in N_2(u)$, if $L_{u \setminus v_1}' \cap L_{u \setminus v_2}' \neq \emptyset$, then we can find a $B_2$}
                \end{figure}



                So, $L_{u \setminus v_1}' \cap L_{u \setminus v_2}' = \emptyset$.
            \end{claimproof}

            By Claim \ref{cla: 323}, we have $|N_2(u)| \leq \frac{d(u) \log n}{n^{\frac{1}{2}}}$.

            Fix a vertex $u$. 
            For each neighbor $v$ of $u$, let $\mathcal{C}_{uv} = \{C_4: C_4 \text{ contains $uv$ as an edge}\}$. 
            Then we have
            \begin{align*}
                \sum_{v \in N(u)} |\mathcal{C}_{uv}| 
                \leq& \sum_{v \in N_1(u)} \left(\max \{d(u), d(v)\} + \frac{n^{\frac{1}{2}}}{\log n} \right) + \sum_{v \in N_2(u)} (d(u) + d(v)) \\
                \leq& \sum_{v \in N_1(u)} \left(d(u) + d(v) - \delta + \frac{n^{\frac{1}{2}}}{\log n} \right) + \sum_{v \in N_2(u)} (d(u) + d(v)) \\
                \leq& \sum_{v \in N(u)} (d(u) + d(v)) + \sum_{v \in N_1(u)} \left(- \delta + \frac{n^{\frac{1}{2}}}{\log n} \right).
            \end{align*}

            Summing over all vertices, we have
            \begin{align*}
                \sum_{u \in V(G)} \sum_{v \in N(u)} |\mathcal{C}_{uv}| 
                \leq \sum_{u \in V(G)} \sum_{v \in N(u)} (d(u) + d(v)) + \sum_{u \in V(G)} \sum_{v \in N_1(u)} \left(- \delta + \frac{n^{\frac{1}{2}}}{\log n} \right).
            \end{align*}

            Since $|N_2(u)| \leq \frac{d(u) \log n}{n^{\frac{1}{2}}}$, we have $|N_1(u)| \geq d(u) (1 - \frac{\log n}{n^{\frac{1}{2}}})$.
            Thus $\sum_{u \in V(G)} \sum_{v \in N_1(u)} (- \delta) \\ \leq 2e (1 - \frac{\log n}{n^{\frac{1}{2}}}) (- \delta)$.
            Thus, we have
            \begin{align*}
                \sum_{u \in V(G)} \sum_{v \in N(u)} |\mathcal{C}_{uv}|
                \leq& 2 \sum_{u \in V(G)} d(u)^2 - 2e(G) \left(1 - \frac{\log n}{n^{\frac{1}{2}}} \right) \delta + 2e(G) \frac{n^{\frac{1}{2}}}{\log n}.
            \end{align*}

            For each $C_4$ $v_1 v_2 v_3 v_4$, there are $4$ choices to fix a vertex $v_i$, and for each $v_i$ there are $2$ choices to choose a neighbor.
            So, each $C_4$ is counted $8$ times.
            Recall the definition of $N(C_4, G)$, which is at most the number of copies of $C_4$ in $G$.
            Thus, we have
            \begin{equation*}
                8 N(C_4, G) \leq 2 \sum_{u \in V(G)} d(u)^2 - 2e(G) \left(1 - \frac{\log n}{n^{\frac{1}{2}}}\right) \delta + 2e(G) \frac{n^{\frac{1}{2}}}{\log n}.
            \end{equation*}

            By Lemma \ref{lem: better lower bound of Bt}, we have
            \begin{equation*}
                4 e(G)^2 -28 n e(G) - \left[\left(\frac{\log n }{n^{\frac{1}{2}}} - 1\right) \delta + \frac{n^{\frac{1}{2}}}{\log n}\right]n e(G) -2 n^3 + o(n^3) \leq 0.
            \end{equation*}
            Solving this inequality, we have $e(G) \leq \frac{-c + \sqrt{c^2 + 32}}{8} (1 + o(1)) n^{\frac{3}{2}}$.
        \end{proof}

        \begin{proof}[Proof of Theorem \ref{thm: upper and lower bound of B2}]
            Let $G$ be a $B_2$-free graph on $n$ vertices.
            Let $c=\frac{3}{\sqrt{11}}$, we have $\frac{-c + \sqrt{c^2 + 32}}{8} = \frac{2c}{3}$.

            Run the following deletion process: when the current graph has $j$ vertices, delete a vertex of degree $< c \sqrt{j}$, if one exists.
            Let the final graph be $G'$ with $n'$ vertices.
            Then we have $e(G) \leq e(G') + c\sum_{j=n'+1}^{n} \sqrt{j} = e(G') + \frac{2}{3} c (n^{\frac{3}{2}} - n'^{\frac{3}{2}}) + O(n)$.
            
            
            If $n' \leq n^{\frac{1}{2}}$, then we have
            $$e(G) \leq e(G') + \frac{2}{3} c (n^{\frac{3}{2}} - n'^{\frac{3}{2}}) + O(n) \leq \binom{n'}{2}+ \frac{2}{3} c (n^{\frac{3}{2}} - n'^{\frac{3}{2}}) + O(n) \leq o(n^{\frac{3}{2}})+\frac{2}{\sqrt{11}}n^{\frac{3}{2}}.$$
            
            If $n'> n^{\frac{1}{2}}$, then since $\delta(G') \geq c(n')^{\frac{1}{2}}$, by Lemma \ref{lem: with large min degree}, we have
            \begin{equation*}
                e(G') \leq \frac{-c + \sqrt{c^2 + 32}}{8} (1 + o(1)) (n')^{\frac{3}{2}} = \frac{2c}{3} (1+o(1)) (n')^{\frac{3}{2}},
            \end{equation*}
            and
            \begin{align*}
                e(G) \leq& e(G') + \frac{2c}{3}(n^{\frac{3}{2}}-n'^{\frac{3}{2}}) + O(n) \\
                \leq& \frac{2c}{3} (1 + o(1)) (n')^{\frac{3}{2}}+\frac{2c}{3}(n^{\frac{3}{2}}-n'^{\frac{3}{2}})+O(n) \\
                \leq& \frac{2c}{3}(1+o(1))n^{\frac{3}{2}} = \frac{2}{\sqrt{11}}(1+o(1))n^{\frac{3}{2}}.
            \end{align*}
                       
            Thus, we complete the proof of the upper bound of $\ex(n,B_2)$.
        \end{proof}

    \section{The proof of Theorem \ref{thm: upper bound of K2 cross Tree}}\label{sec: tree}

        In this section, we will give the proof of Theorem~\ref{thm: upper bound of K2 cross Tree}.
        In order to find a $K_2 \mathbin{\square} T$, we would like to assume that for each edge $xy$, there are at least $t$ neighbors $z$ of $y$ such that $d(x, z) \geq t$. We will use a deletion procedure which was employed in~\cite{bradavc2023turan}, but we will aim for a small leading coefficient through a careful analysis. 
        The following lemma allows us to find a subgraph with the property mentioned above. 

\begin{lemma} \label{embedlemma}
Let $T$ be a tree on $t$ vertices and $G$ be a bipartite graph with parts~$X$ and~$Y$ and at least one edge. Assume that for every edge $xy \in E(G)$ with $x\in X$ and $y\in Y$ the vertex $y$ has at least $t$ neighbors $z$ with $d(\{x,z\}) \ge t$. Then $G$ contains $K_2 \mathbin{\square} T$ as a subgraph.
\end{lemma}

\begin{proof}
Take an ordering $v_1,v_2,\dots,v_t$ of the vertices of $T$ so that for each $v_i$ with $i>1$ we have that $v_i$ is adjacent to exactly one vertex of $\{v_1,v_2,\dots,v_{i-1}\}$ in $T$.  Suppose $K_2 \mathbin{\square} T[\{v_1,v_2,\dots,v_s\}]$ has been embedded and $v_{s+1}$ is adjacent to $v_r$, $r<s$, in $T$. At this stage $s$ vertices of $X$ and $s$ vertices of $Y$ have appeared in the embedding. Let $x_r y_r$ be the edge corresponding to $v_r$ in the embedding in $G$.  Then since $s<t$, $y_r$ has a neighbor $z$ in $X$ which is not yet present in the embedding, and $d(x_r,z)\ge t$ so $x_r$ and $z$ have a common neighbor $y$ which is not yet present in the embedding.  Extending the embedding of $K_2 \mathbin{\square} T[\{v_1,v_2,\dots,v_s\}]$ with the edges $x_r y$, $zy$, $z y_r$ yields an embedding of  $K_2 \mathbin{\square} T[\{v_1,v_2,\dots,v_{s+1}\}]$.
\end{proof}

        We restate Theorem~\ref{thm: upper bound of K2 cross Tree} for convenience.

        \begin{theorem}\label{thm: upper bound of K2 cross Tree Restate}
            Let $T$ be a tree with $t$ vertices.
            Then
            $$\ex_{bip}(n,K_2 \mathbin{\square} T)\leq \frac{\sqrt{t-1}}{2\sqrt{2}}(1+o(1)) n^{\frac{3}{2}}.$$
        \end{theorem}

\begin{proof}
Let $G$ be an $n$-vertex bipartite graph with parts $X$ and $Y$ with $|X|\le |Y|$. Let $T$ be a tree on $t$ vertices. Assume that $G$ does not contain $K_2 \mathbin{\square} T$ as a subgraph and set $m= e(G)$. We will construct a new graph via a sequence of deletions of the following types: 

\smallskip
\emph{(Type 1)} If there exists $y\in Y$ with $d(y) \le t-1$, delete $y$ and all its incident edges.

\smallskip
\emph{(Type 2)} If there is no Type~1 deletion available and there exists an edge $xy$ with $x\in X$ and $y\in Y$ such that $y$ has at most $t-1$ neighbors $z$ with $d(\{x,z\})\ge t$, then delete $xy$.

\smallskip
Apply deletions of Type~1 and Type~2 to $G$ until no further deletion is possible, and call the resulting graph $G'$. Any remaining edge satisfies the condition of Lemma~\ref{embedlemma}. Thus, since $G$ contains no $K_2 \mathbin{\square} T$, we have that $G'$ is empty.

For each Type~2 deletion of an edge $xy$, let $d$ be the degree of $y$ immediately before deleting $xy$.  Then among the $d-1$ neighbors of $y$ in $N(y)\setminus \{x\}$, at least $(d-1)-(t-1)=d-t$ vertices $z$ satisfy $d(\{x,z\})\le t-1$. For each Type~2 deletion of an edge $xy$ we record all of the triples $(x,z,y)$ where $d(\{x,z\})\le t-1$. Let $\mathcal{S}$ be the set of these triples. 

Observe that at each Type~2 deletion of an edge $xy$, the codegree of each pair $x$ and $z$ is reduced by exactly one. The initial codegree of a pair $x$ and $z$ (the codegree at the first moment a pair $\{x, z\}$ is recorded) considered is at most $t-1$, thus the total number of times $x$ and $z$ can occur together in a triple from $\mathcal{S}$ (so either as $(x,z,y)$ or $(z,x,y)$) is at most $t-1$. Thus,
\begin{equation}\label{uppereq}
|\mathcal{S}|\le (t-1)\binom{|X|}{2}.
\end{equation}

We will now obtain a lower bound on $|\mathcal{S}|$. Since $G'$ is empty and vertices $y$ are only deleted when their degree is at most $t-1$, it must happen that the degree of each vertex $y$ of degree at least $t$ is reduced from $d_G(y)$ to $t-1$ through Type~2 deletions of edges incident to $y$. We have seen that when $y$ has degree $d$ such a step adds at least $d-t$ new triples to $\mathcal{S}$. Summing, we obtain that the number of triples $(x,z,y)$ involving $y$ with $d_G(y)\ge t$ is at least

\[
(d_G(y)-t) + (d_G(y)-t-1) + \cdots + 0 \;=\; \binom{d_G(y)-t+1}{2},
\]
and $0$ if $d_G(y)\le t-1$. Thus,
\begin{equation} \label{binbound}
|\mathcal{S}| \ge \sum_{y\in Y} \binom{\max\{d_G(y)-t+1,0\}}{2}.
\end{equation}
Now set $a_y\coloneqq \max\{d_G(y)-t+1,0\}$.
Then $\binom{a_y}{2} = \tfrac12(a_y^2-a_y)$, so~\eqref{binbound} implies
\begin{equation}\label{lowerest}
|\mathcal{S}|
\;\ge\;
\frac{1}{2}\sum_{y\in Y} a_y^2 \;-\; \frac{1}{2}\sum_{y\in Y} a_y.
\end{equation}
Clearly $\sum_y a_y \le \sum_y d_G(y) = m$. On the other hand, 
\[
\sum_y a_y \;\ge\; \sum_y d_G(y) - (t-1)|Y| \;=\; m - (t-1)|Y|,
\]
and also $\sum_y a_y \ge 0$. By the Cauchy-Schwarz inequality 
\[
\sum_{y\in Y} a_y^2 \;\ge\; \frac{(\sum_{y\in Y} a_y)^2}{|Y|}
\;\ge\; \frac{(\max\{m-(t-1)|Y|,0\})^2}{|Y|}.
\]
If $m \le (t-1)|Y|$, then $m=O(n)$. So we assume $m>(t-1)|Y|$.  Plugging the above estimates into~\eqref{lowerest} we get
\begin{equation}
\label{lowereq}
|\mathcal{S}|
\;\ge\;
\frac{\bigl(m-(t-1)|Y|\bigr)^2}{2|Y|} \;-\; \frac{m}{2}.
\end{equation}
Combining~\eqref{uppereq} and~\eqref{lowereq} yields
\[
\frac{(m-(t-1)|Y|)^2}{2|Y|} \;-\; \frac{m}{2}
\;\le\;
(t-1)\binom{|X|}{2}.
\]
Using the well-known result that forbidding a bipartite graph implies $m=o(n^2)$ (or the stronger estimate $m=O(n^{\frac{3}{2}})$ from~\cite{bradavc2023turan}) and the fact that $|X|$ and $|Y|$ are at most $n$, we obtain
\[
m^2  \;\le\; (t-1)|Y||X|^2 +o(n^3).
\]
By our assumption that $|X|\le |Y|$, it follows that $|Y||X|^2 \le \frac{n^3}{8}$ and thus
\[
m \;\le\; \frac{\sqrt{t-1}}{2\sqrt{2}}n^{\frac{3}{2}} + o(n^{\frac{3}{2}}).\qedhere
\]
\end{proof}

With a simple modification of the embedding procedure from Lemma~\ref{embedlemma} we can obtain the following stronger lemma.

\begin{lemma} \label{generallemma}
Let $F$ be a bipartite graph with parts $A$ and $B$ constructed in the following way.  Beginning with an edge, apply any sequence of operations of the following type until $|B|=t$: 

\smallskip
(Operation) Take any edge $ab$ with $a\in A$ and $b\in B$ created so far, and add two new vertices $c \in A$, $d \in B$ and the edges $bc$, $cd$ and $ad$.  Then add any number $r\ge 0$  vertices $f \in B$ and the edges $af$, $fc$.  
\smallskip

Let $G$ be a bipartite graph with parts $X$ and $Y$ and at least one edge. Assume that for every edge $xy \in E(G)$ with $x\in X$ and $y\in Y$ the vertex $y$ has at least $t$ neighbors $z$ with $d(\{x,z\}) \ge t$. Then $G$ contains $F$ as a subgraph.
\end{lemma}

From Lemma~\ref{generallemma} we obtain the following result.

\begin{theorem}\label{generalversion}
Let $F$ be a bipartite graph constructed as in Lemma~\ref{generallemma}. Then 
\[
\ex_{bip}(n,F)\leq \frac{\sqrt{t-1}}{2\sqrt{2}}(1+o(1)) n^{\frac{3}{2}}.
\]
\end{theorem}
\begin{proof}
    The proof is exactly the same as the proof of Theorem~\ref{thm: upper bound of K2 cross Tree Restate}, using Lemma~\ref{generallemma} in place of Lemma~\ref{embedlemma}.
\end{proof}

The graphs $K_2 \mathbin{\square} T$ and $K_{2,t}$ satisfy the conditions of the theorem as do many other graphs.  Recall that $\ex_{bip}(n,K_{2,t}) = \frac{\sqrt{t-1}}{2\sqrt{2}}(1+o(1)) n^{\frac{3}{2}}$~\cite{kst,mors1981new}, so the theorem is best possible over this wider class of graphs.

    \section{Proof of Theorem \ref{thm: bipartite upper and lower bound of B2}}\label{sec: Bt bipartite}

        In this section, we will give the proof of Theorem \ref{thm: bipartite upper and lower bound of B2}.
        Recall that the statement of Theorem \ref{thm: bipartite upper and lower bound of B2} is $\ex_{bip} (n, B_2) \leq 0.468 (1+o(1)) n^{\frac{3}{2}}$.
        In fact, for any given $t$, a better upper bound for $\ex_{bip} (n, B_t)$ can be obtained using the following analysis.

        \begin{lemma}\label{lem: B2 free bipartite}
            Let $G$ be a $B_2$-free bipartite graph on $n$ vertices.
            If $\delta(G)\geq cn^{\frac{1}{2}}$ for a constant $c$, then for every $\epsilon>0$ we have
            $$e(G)\leq \max\left\{\sqrt{\frac{4}{27}+2\epsilon},\sqrt{\frac{1}{4}-4c^2\epsilon}\right\}n^{\frac{3}{2}}+o(n^{\frac{3}{2}}).$$
        \end{lemma}
        \begin{proof}
            Let $G$ be a $B_2$-free bipartite graph with parts $V_1$ and $V_2$, where $|V_1|=n_1$, $|V_2|=n_2$ and $n=n_1+n_2$.
            Without loss of generality, assume that $|V_1| \leq |V_2|$.
            Choose an edge $e=v_1v_2$ where $v_i\in V_i$ for $i=1,2$. 
            Let $\mathcal{C}_i$ be the collection of $C_4$ (denoted by $C$) such that $C \cap V_i$ is a good set.
            Then $N_i(C_4, G) = |\mathcal{C}_i|$.
            The number of copies of $C_4$ in $\mathcal{C}_i$ containing $e$ is at most $d(v_i)$.
            Then for all edges $e$ containing $v_i$, we sum up all the $C_4$ in $\mathcal{C}_i$ containing $e$, and have $\sum_{u\in N(v_i)}d(v_i)=d(v_i)^2$.
            Summing over all vertices in $V_i$, each $C_4$ in $\mathcal{C}_i$ is counted four times. Thus, we have
            $$4N_i(C_4, G)\leq \sum_{v\in V_i}d(v)^2.$$
            Now we improve the estimate of the upper bound for $N_i(C_4, G)$.
            Notice that if there exists a good set $\{u_1,u_2\}$ in $V_1$, then for every $v\in N(u_1)\cap N(u_2)$, and every $w\in V_2\setminus \{v\}$ with $\{v,w\}$ good, $\{u_1,u_2\}\cap (N(v)\cap N(w))=\emptyset$.
            Otherwise, there exists a $B_2$ in $G$.
            As a result, there is no $C_4$ in $\mathcal{C}_2$ containing either $\{v, u_1\}$ or $\{v, u_2\}$.
            Thus there are $2\sum_{v\in N(u_1)\cap N(u_2)}d(v)\geq 2cn^{\frac{1}{2}}d(u_1,u_2)$ missing $C_4$ in $\mathcal{C}_2$.
            A fixed edge $vu$ can belong to at most one good set $\{u, u'\}$ with $v \in \{u, u'\}$.
            Otherwise, if $vu$ is in two good sets $\{u, u'\}$ and $\{u, u''\}$, one can choose distinct $w \in N(u, u') \setminus \{v\}$ and $z \in N(u, u'') \setminus \{v\}$, and the $u v u' w$ and $u v u'' z$ share the edge $uv$, yielding a $B_2$.
            Hence, we have
            \begin{equation}\label{eq: bip new uppwer bound of Q2}
                4N_2(C_4, G)\leq \sum_{v\in V_2} d(v)^2- 2cn^{\frac{1}{2}} \sum_{\substack{\{u_1,u_2\}\subseteq V_1, \\\{u_1,u_2\}\text{ is good}}}d(u_1,u_2).
            \end{equation}

            If $N_1(C_4, G)\leq \epsilon n^2$ for a fixed constant $0<\epsilon$, which we will determine later, then by Lemma~\ref{lem: lower bound of number of C4 for Bt on bip}, we have
            $$\sum_{v\in V_1}d(v)^2\leq 3e(G)+n_2^2+2\epsilon n^2.$$
            It implies that
            $$\frac{(e(G))^2}{n_1}\leq \sum_{v\in V_1}d(v)^2\leq 3e(G)+n_2^2+2\epsilon n^2.$$
            Then we have
            $$e(G)\leq \sqrt{n_1n_2^2+2n_1\epsilon n^2}+O(n^{\frac{5}{4}})\leq \sqrt{\frac{4}{27}+2\epsilon}(1+o(1))n^{\frac{3}{2}}.$$
            Then we may assume $N_1(C_4, G)>\epsilon n^2$.
            Note that if $\{u_1,u_2\}$ is good, and $v_1,v_2\in N(u_1)\cap N(u_2)$, then $N(v_1)\cap N(v_2)=\{u_1,u_2\}$.
            Otherwise, there exists a $B_2$ in $G$.
            As a result, $d(u_1,u_2)\leq \frac{n_2}{cn^{\frac{1}{2}}}$ for every good set $\{u_1,u_2\}\subseteq V_1$.
            Notice that
            $$\epsilon n^2<N_1(C_4, G)=\sum_{\substack{\{u_1,u_2\}\subseteq V_1, \\ \{u_1,u_2\}\text{ is good}}}\binom{d(u_1,u_2)}{2}\leq \frac{n_2}{2cn^{\frac{1}{2}}}\sum_{\substack{\{u_1,u_2\}\subseteq V_1, \\  \{u_1,u_2\}\text{ is good}}}d(u_1,u_2).$$
            It implies that
            \begin{equation}\label{eq: bio lower bound of good pairs}
                \sum_{\substack{\{u_1,u_2\}\subseteq V_1, \\ \{u_1,u_2\}\text{ is good}}}d(u_1,u_2)\geq \frac{2c\epsilon n^{\frac{5}{2}}}{n_2}.
            \end{equation}
            Then, by Equation (\ref{eq: bip new uppwer bound of Q2}) and Equation (\ref{eq: bio lower bound of good pairs}), we have
            $$4N_2(C_4, G)\leq \sum_{v\in V_2} d(v)^2- \frac{4c^2\epsilon n^{3}}{n_2}.$$
            Combining with Lemma \ref{lem: lower bound of number of C4 for Bt on bip}, we have
            $$\sum_{v\in V_2}d(v)^2\leq 3e(G)+n_1^2+\frac{1}{2} \left(\sum_{v\in V_2}d(v)^2- \frac{4c^2\epsilon n^{3}}{n_2} \right).$$
            It implies that
            $$e(G)\leq (1+o(1))\sqrt{2n_1^2n_2-4c^2\epsilon n^3}\leq \left(\sqrt{\frac{1}{4}-4c^2\epsilon}\right) (1+ o(1)) n^{\frac{3}{2}},$$
            and the second inequality follows from $n_1 \leq n_2$ and $n_1 + n_2 = n$.
            We complete the proof.
        \end{proof}

        \begin{proof}[Proof of Theorem \ref{thm: bipartite upper and lower bound of B2}]
            Let $G$ be a $B_2$-free bipartite graph.
            Delete vertices with degree less than $cn^{\frac{1}{2}}$.
            If the remaining graph $G'$ contains at most $n^{\frac{1}{2}}$ vertices, then the total number of edges $e(G) \leq \binom{n^{\frac{1}{2}}}{2} + c (n-n^{\frac{1}{2}})n^{\frac{1}{2}} < cn^{\frac{3}{2}} + O(n)$.

            Then we assume that the remaining graph $G'$ contains at least $n' = n^{\frac{1}{2}}$ vertices.
            By Lemma \ref{lem: B2 free bipartite}, we have $e(G') \leq \max\left\{\sqrt{\frac{4}{27}+2\epsilon},\sqrt{\frac{1}{4}-4c^2\epsilon}\right\}(n')^{\frac{3}{2}}+o((n')^{\frac{3}{2}}).$

            If there are positive constants $\epsilon$ and $c$ such that
            $$c \geq \max\left\{ \sqrt{\frac{4}{27}+2\epsilon},\sqrt{\frac{1}{4}-4c^2\epsilon}\right\},$$
            then we have $e(G')\leq c(n')^{\frac{3}{2}}+o((n')^{\frac{3}{2}})$ and
            \begin{equation*}
                e(G) \leq e(G') + c(n - n')n^{\frac{1}{2}} \leq c(n')^{\frac{3}{2}}+o((n')^{\frac{3}{2}}) + c(n - n')n^{\frac{1}{2}} \leq cn^{\frac{3}{2}} + o(n^{\frac{3}{2}}).
            \end{equation*}
            By taking $\epsilon = 0.0354$, we have $c\geq  0.468$, i.e., $e(G) \leq 0.468 (1+o(1)) n^{\frac{3}{2}}$.
        \end{proof}


\section{Acknowledgements}

    The research of Zhao is supported by the China Scholarship Council (No. 202506210250) and the National Natural Science Foundation of China (Grant 12571372).

    The research of Cheng is supported by the National Natural Science Foundation of China (Nos. 12131013 and 12471334), Shaanxi Fundamental Science Research Project for Mathematics and Physics (No. 22JSZ009) and the China Scholarship Council (No. 202406290241).

    The research of Chi is supported by the China Scholarship Council (No. 202406140160).

    The research of Gy\H{o}ri and Tompkins was supported by NKFIH grant K132696.

    The research of Wang is supported by the China Scholarship Council (No. 202506210200) and the National Natural Science Foundation of China (Grant 12571372).
    \bibliographystyle{amsplain}
    \bibliography{references}
\end{document}